\documentclass[12pt]{amsart}

\usepackage{amsmath}
\usepackage{amssymb}
\usepackage{amsthm}
\usepackage{pstricks}
\usepackage{pst-node}
\usepackage{pst-plot}
\usepackage{subfigure}
\usepackage{graphicx}
\usepackage{tabmac}
\usepackage[normalem]{ulem}
\numberwithin{equation}{section}

\newtheorem{thm}{Theorem}[section]
\newtheorem{prop}[thm]{Proposition}

\newtheorem{lem}[thm]{Lemma}
\newtheorem{cor}[thm]{Corollary}
\newtheorem*{theorem*}{Theorem}

\theoremstyle{definition}

\newtheorem{rmk}[thm]{Remark}
\newtheorem{ex}[thm]{Example}

\newcommand{\jdt}{\operatorname{jdt}}

\def\C{{\mathcal C}}
\def\B{{\mathcal {BK}}}
\def\q{{\mathfrak {q}}}

\def\SSYT{{\textrm{SSYT}}}
\def\RSK{{\textrm{RSK}}}
\def\Rect{{\textrm{Rect}}}
\def\Evac{{\textrm{Evac}}}
\def\Boom{{\textrm{Boom}}}

\def\cellsize{.2}
\newcommand{\cell}[2]{
\psset{unit=\cellsize}
\rput(#2,-#1){\pspolygon(0,0)(0,1)(-1,1)(-1,0)}
\psset{unit=1cm}
}

\newenvironment{shape}[2]
{\psset{unit=\cellsize}
\begin{pspicture}(0,-#1)(#2,0) 
\psset{unit=1cm}
}
{\end{pspicture}}

\begin{document}
\title{The Berenstein-Kirillov group and cactus groups}
\author{Michael Chmutov}
\address{Department of Mathematics, University of Minnesota, Minneapolis, MN 55455, USA}
\author{Max Glick}
\address{Department of Mathematics, University of Connecticut, Storrs, CT 06269, USA}
\author{Pavlo Pylyavskyy}
\address{Department of Mathematics, University of Minnesota, Minneapolis, MN 55455, USA}
\thanks{M.C. was partially supported by NSF grant DMS-1503119. M. G. was partially supported by NSF grant DMS-1303482.  P. P. was partially supported by NSF grants DMS-1148634, DMS-1351590, and Sloan Fellowship.}

\begin{abstract}
Berenstein and Kirillov have studied the action of Bender-Knuth moves on semistandard tableaux. Losev has studied a cactus group action in Kazhdan-Lusztig theory; in type $A$ this action can also be identified in the work of Henriques and Kamnitzer. 
We establish the relationship between the two actions. We show that the Berenstein-Kirillov group is a quotient of the cactus group. We use this to derive previously unknown relations in the Berenstein-Kirillov group. 
We also determine precise implications between subsets of relations in the two groups, which yields a presentation for cactus groups in terms of Bender-Knuth generators.
\end{abstract}

\maketitle

\setcounter{tocdepth}{1}
\tableofcontents

\section{Introduction}

\subsection{Cactus groups}

Let us briefly recall the origins of cactus groups, as introduced by Henriques and Kamnitzer in \cite{HK}, see also \cite{De, DJS}. 

Coboundary categories were introduced by Drinfeld in \cite{Dr}. They are monoidal categories equipped with a {\it {commutor}} $$\sigma_{A,B} \colon A \otimes B \longrightarrow B \otimes A,$$ which satisfies two conditions:
\begin{enumerate}
 \item The commutor is an involution, i.e. $$\sigma_{B,A} \circ \sigma_{A,B} = I.$$
 \item One has $$\sigma_{A, C \otimes B} \circ (I \otimes \sigma_{B,C}) = \sigma_{B \otimes A, C} \circ (\sigma_{A,B} \otimes I).$$
\end{enumerate}
Coboundary categories are analogous to braided monoidal categories, with the second relation above playing the role of the braid relation. If one has a long tensor product in a coboundary category, it is natural to apply {\it {interval reversal}} operators to it. Specifically, for any $1 \leq i < j \leq n$
one can apply a reversal operator to the corresponding interval between $i$th and $j$th factors:
$$\q_{[i,j]} \colon A_1 \otimes \dotsc \otimes A_n \longrightarrow A_1 \otimes \dotsc \otimes A_{i-1} \otimes A_j \otimes \dotsc \otimes A_i \otimes A_{j+1} \otimes \dotsc \otimes A_n.$$
Henriques and Kamnitzer in \cite{HK} introduce the {\it {cactus group}} as the group of relations satisfied by interval reversals $\q_{[i,j]}$ in a coboundary category. Formally, the cactus group $\C_n$ is generated by $\q_{[i,j]}$, $1 \leq i < j \leq n$ subject to relations 
\begin{enumerate}
 \item $\q_{[i,j]}^2 = I,$
 \item $\q_{[i,j]} \q_{[k,l]} = \q_{[k,l]} \q_{[i,j]} \text{ if } j<k,$
 \item $\q_{[i,j]} \q_{[k,l]} \q_{[i,j]} = \q_{[i+j-l,i+j-k]} \text{ if } i \leq k < l \leq j.$
\end{enumerate}

\subsection{A theorem of Losev}

Davis, Januszkiewicz, and Scott in \cite{DJS} associate a group $\C_W$ to any finite Weyl group $W$  so that for $W$ of type $A$ the group $\C_W$ coincides with the Henriques-Kamnitzer cactus group defined above.  
The term {\it {mock reflection group}} is used in \cite{DJS} instead of {\it {cactus group}}. 

Losev defines an action of $\C_W$ on $W$ which satisfies the following two properties. 
\begin{itemize}
 \item The action preserves right Kazhdan-Lusztig cells and permutes the left ones.
 \item The action respects restriction to parabolic subgroups.
\end{itemize}
We refer the reader to \cite[Theorem 1.1]{L} for a precise statement of Losev's result. Of interest for us is a reformulation of his result in type $A$ which we proceed to describe. 

First, recall the notions of {\it {Robinson-Schensted-Knuth insertion}} and of {\it {Sch\"utzenberger involution}}, which are fundamental constructs in algebraic combinatorics. The RSK insertion is a map 
$$\RSK \colon w \mapsto (P,Q)$$ that maps a permutation $w \in S_n$ to a pair of standard Young tableaux $(P,Q)$. The Sch\"utzenberger involution $$\Evac \colon P \mapsto P'$$ maps a standard Young tableau $P$ to a different standard Young tableau $P'$ of the same shape. 
Both of those notions are classical and we refer the reader for example to Stanley's book \cite{St} for an accessible introduction and precise definitions. 

Now, the type $A$ case of the theorem of Losev can be stated as follows. Define operators $\tau_{[i,j]}$ acting on permutations $w\in S_n$ as follows: 
$$\tau_{[i,j]}(w) = w {|_{w[i,j] \leftarrow \RSK^{-1}((id \times \Evac)(\RSK(w[i,j])))}}$$ where $w[i,j]$ denotes the interval of $w$ between $i$th and $j$th positions. In other words, 
\begin{itemize}
 \item we restrict $w$ to the window between $i$th and $j$th positions,
 \item insert the resulting word using RSK,
 \item apply the Sch\"utzenberger involution to the $Q$ part of the resulting pair $(P,Q)$,
 \item apply reverse RSK to the resulting pair $(P,\Evac(Q))$,
 \item and finally replace the original interval $w[i,j]$ with the result.
\end{itemize}

\begin{thm}[\cite{L}] \label{thm:Losev}
 The operators $\tau_{[i,j]}$ satisfy the relations of the cactus group $\C_n$. 
\end{thm}

\begin{ex}
 When $j-i=2$, the operators $\tau_{[i,j]}$ act as Knuth relations, or equivalently as the $*$ operator from Kazhdan-Lusztig theory. For example, if we have $a = w_i < b = w_{i+2} < c = w_{i+1}$, then the RSK insertion of $w[i,i+2]=acb$ looks as follows:
 \begin{displaymath}
\RSK(w[i,i+2]) = (P,Q) = \left(\tableau[sY]{a, b\\ c}, \tableau[sY]{1, 2\\ 3} \right).
\end{displaymath}
 The Sch\"utzenberger involution applied to $Q$ gives 
  \begin{displaymath}
\Evac(Q) = \tableau[sY]{1, 3\\ 2}.
\end{displaymath}
Finally, $\RSK^{-1}(P,\Evac(Q)) = cab$. In other words, the action of $\tau_{[i,i+2]}$ on $w$ is as follows:
$$\tau_{[i,i+2]} \colon (\cdots acb \cdots)  \mapsto (\cdots cab \cdots) $$ which one immediately recognizes as a Knuth move applied to three consecutive positions of $w$. 
\end{ex}

\subsection{The Berenstein-Kirillov group}
{\it {Semistandard Young tableaux}} are also fundamental in algebraic combinatorics. They are fillings of Young diagrams with positive integers that weakly increase in rows and strictly increase  in columns.  

{\it {Bender-Knuth moves}} are certain involutions on semistandard tableaux. For each $i$ one has a Bender-Knuth move $t_i$ that acts on $i$'s and $(i+1)$'s only in a semistandard tableau, fixing the rest of the filling.
The reader can find details in Section \ref{sec:semistandard}. We also refer the reader to \cite{St} for more background on semistandard tableaux. 

In \cite{BK} Berenstein and Kirillov study relations satisfied by the $t_i$. The following object arises in their study: consider the free group generated by the $t_i$, $i \in \mathbb Z_{>0}$, modulo the relations satisfied by the $t_i$ when acting on all semistandard tableaux of all 
possible shapes. We call the resulting group $\B$ the {\it {Berenstein-Kirillov group}}, or the {\it {BK group}} for brevity. 

\begin{rmk}
The definition we give here uses {\it {combinatorial}} Bender-Knuth moves. One can extend the action of the $t_i$ to {\it {Gelfand-Tsetlin patterns}} filled with not necessarily integer entries, thus obtaining {\it {piecewise-linear}} Bender-Knuth moves. This is done in 
\cite{BK}. Finally, this piecewise-linear action can also be lifted to a birational setting in a natural way, providing {\it {birational}} Bender-Knuth moves. There are respectively three versions of the BK group, depending on which of the three actions one chooses to use. 
We expect the three groups to coincide, however we do not know if this is the case. In this paper we stick to the combinatorial BK group defined above.
\end{rmk}

While well-defined, the BK group does not come with an explicit set of relations. In \cite{BK} Berenstein and Kirillov list the following relations that hold among the $t_i$ in $\B$:
\begin{enumerate}
 \item $t_i^2= I, \;\; t_i t_j = t_j t_i \text{ if } |i-j|>1;$
 \item $(t_1 q_i)^4 = I \text{ where } q_i = t_1 (t_2 t_1) \dotsc (t_i t_{i-1} \dotsc t_1) \text{ and } i \geq 3;$
 \item $(t_1 t_2)^6 = I.$
\end{enumerate}

\subsection{The main results}

The main goal of this paper is to establish a precise relation between the BK group and the cactus group. In particular, we
{{use the relations of the BK group to reprove and generalize Losev's result}} in type $A$. The resulting theorem can also be deduced from a special case of the theory built by Henriques and Kamnitzer in \cite{HK}, thus we attribute this theorem to them below. 
We also use cactus groups to extend the list of known relations in $\B$, and we study which relations of the latter are equivalent to which relations of the former. 

Thus, the main results of this paper are Theorems \ref{thm:map}, \ref{thm:new_relations}, \ref{thm:rel1}, and \ref{thm:rel2} below. Even though Corollary \ref{cor:tau} is implicit in \cite{HK}, we believe our proof using growth diagram techniques is interesting enough to be 
also considered one of the main results of this paper. 

For $i<j$ define operators $q_{[i,j]}$ in $\B$ as follows:
$$q_{[i,j]} = q_{j-1} q_{j-i} q_{j-1},$$
where the $q_i$ are defined as before:
$$q_i = t_1 t_2 t_1 \dotsc t_i t_{i-1} \dotsc t_1.$$
Denote $\B_n$ the subgroup of $\B$ generated by the first $n-1$ of the $t_i$'s. The first main result of our paper is as follows. 

\begin{thm} \label{thm:map}
 The map $\q_{[i,j]} \mapsto q_{[i,j]}$ is a group homomorphism from the cactus group $\C_n$ to $\B_n$. 
\end{thm}

Extend the action of operators $\tau_{[i,j]}$ from permutations to {\it {words}}, as explained in Section \ref{sec:words}. The following corollary is the generalization of Losev's result \cite[Theorem 1.1]{L} in type $A$. 
It is also implicit in the work of Henriques and Kamnitzer, to whom we attribute it. 

\begin{cor}[\cite{HK}] \label{cor:tau}
 The operators $\tau_{[i,j]}$ satisfy the relations of the cactus group $\C_n$.
\end{cor}

Our next result extends the list of known relations of the BK group.

\begin{thm} \label{thm:new_relations}
 The following relations are satisfied in $\B$ for $i+1 < j < k$:
 $$(t_i q_{[j,k]})^2 = I.$$
\end{thm}

This is a direct generalization of the relation $(t_1 q_i)^4 = I$ found by Berenstein and Kirillov. Indeed, $$(t_1 q_i)^4 = (t_1 q_i t_1 q_i)^2 = (t_1 q_{[i,i+1]})^2.$$

Finally, one can make a more refined statement about which relations in $\C_n$ imply which relations in $\B_n$, and the other way around. Consider a free group generated by $t_i$, $i \in \mathbb Z_{>0}$, and another free group generated by $q_{[i,j]}$, $1 \leq i < j$. Consider the morphisms between those groups given by 
$$q_{[i,j]} \mapsto q_{j-1} q_{j-i} q_{j-1}, \text{ where } q_i = t_1 t_2 t_1 \dotsc t_i t_{i-1} \dotsc t_1$$ in one direction, and 
$$t_1 \mapsto q_{[1,2]}, \;\; t_2 \mapsto q_{[1,2]}q_{[1,3]}q_{[1,2]}, \;\; t_i \mapsto q_{[1,i]}q_{[1,i+1]}q_{[1,i]}q_{[1,i-1]} \text{ for } i>2$$ in the other direction. 

\begin{thm} \label{thm:rel1}
 The relations $$t_i^2=I, \;\; t_i t_j = t_j t_i \text{ if } |i-j|>1$$  are equivalent to the relations $$q_{[i,j]}^2 = I, \;\; q_{[i,j]} q_{[k,l]} q_{[i,j]} = q_{[i+j-l,i+j-k]} \text{ if } i \leq k < l \leq j.$$  More precisely, 
\begin{multline*}
\left\langle t_i| t_i^2=I, t_i t_j = t_j t_i \text{ if } |i-j|>1\right\rangle\cong\\
   \left\langle q_{[i,j]}| q_{[i,j]}^2 = I, q_{[i,j]} q_{[k,l]} q_{[i,j]} = q_{[i+j-l,i+j-k]} \text{ if } i \leq k < l \leq j
	     \right\rangle
\end{multline*}
and the above morphisms descend to an isomorphism (and its inverse).
\end{thm}

\begin{thm} \label{thm:rel2}
 The relations $$t_i^2=I, \;\; t_i t_j = t_j t_i \text{ if } |i-j|>1 \text{  and  } (t_i q_{k-1}q_{k-j}q_{k-1})^2 = I \text{ for } i+1 < j < k$$ are equivalent to the relations 
 $$q_{[i,j]}^2 = I, \;\; q_{[i,j]} q_{[k,l]} q_{[i,j]} = q_{[i+j-l,i+j-k]} \text{ if } i \leq k < l \leq j,$$ $$\text{  and  } q_{[i,j]} q_{[k,l]} = q_{[k,l]} q_{[i,j]} \text{ if } j<k.$$
\end{thm}

This means that the generators $t_i$, $i = 1, \ldots, n-1$ with relations 
$$t_i^2=I, \;\; t_i t_j = t_j t_i \text{ if } |i-j|>1 \text{  and  } (t_i q_{k-1}q_{k-j}q_{k-1})^2 = I \text{ for } i+1 < j < k$$
give an {\it {alternative presentation}} for the cactus group $\C_n$. It is natural to refer to the $t_i$ in this context as {\it {Bender-Knuth generators}} of the cactus group.

\begin{rmk}
 We see that $$(t_1 t_2)^6 = I$$ is the only known relation in $\B$ which does not follow from the cactus group relations. Its nature remains mysterious, as does the question of whether there are any other undiscovered relations in the BK group which do not 
 follow from the relations in the cactus group.
\end{rmk}

\begin{rmk}
Some of the results from this paper have been proven independently, and at roughly the same time, by Berenstein and Kirillov \cite{BK2}.
\end{rmk}

\medskip
\textbf{Acknowledgments.} We are grateful to Ivan Losev for attracting our attention to his work and for his comments on a previous draft of this paper.  We thank Jake Levinson for informative discussions and Darij Grinberg for numerous comments and reference suggestions.

\section{Background}
\subsection{Standard Young tableaux}
Let $\lambda$ be a partition of $n = |\lambda|$.  A \emph{standard Young tableau} of shape $\lambda$ is a filling of the Young diagram for $\lambda$ with the numbers $1,2,\ldots, n$ with the entries increasing along rows and down columns.  For example,
\begin{displaymath}
\tableau[sY]{ 1, 3, 4, 5\\ 2, 6, 8, 9 \\7}
\end{displaymath}
is a standard Young tableau of shape $441$. For a standard Young tableau $T$, we denote its shape by $\operatorname{sh}(T)$. A standard Young tableau $T$ can also be represented as a saturated chain 
\begin{displaymath}
\emptyset = \lambda^{(0)} \subset \lambda^{(1)} \subset \ldots \subset \lambda^{(n)} = \operatorname{sh}(T)
\end{displaymath}
in Young's lattice where $\lambda^{(k)}$ is the shape consisting of those boxes of $T$ with entry at most $k$.  The preceding example corresponds to the chain
\begin{displaymath}
\emptyset \subset 1 \subset 11 \subset 21 \subset 31 \subset 41 \subset 42 \subset 421 \subset 431 \subset 441.
\end{displaymath}

If $\mu \subseteq \nu$ are a pair of shapes one also speaks of standard Young tableaux of skew-shape $\nu / \mu$.  These are fillings of $\nu / \mu$ by the numbers $1,2,\ldots, |\nu| - |\mu|$ with the entries increasing along rows and down columns.  Such a tableau can be represented as a saturated chain of shapes
\begin{displaymath}
\mu = \lambda^{(0)} \subset \lambda^{(1)} \subset \ldots \subset \lambda^{(|\nu|-|\mu|)} = \nu.
\end{displaymath}

We will use heavily two operations on tableaux, namely Sch\"utzenberger involution and jeu-de-taquin.  Following \cite[Chapter 7, Appendix 1]{St}, we define both in terms of a local growth rule.  Let $\lambda, \mu, \nu$ be shapes such that $\nu \subset \mu \subset \lambda$ where each larger shape is obtained from the previous one by adding a single box, say $\mu = \nu \cup box_1$ and $\lambda = \mu \cup box_2$.  The \emph{growth rule} applied to this triple replaces $\mu$ with $\nu \cup box_2$ if this is in fact a shape, and otherwise does nothing.  The result, whether different from $\mu$ or the same, is denoted $\mu'$.  The rule is depicted pictorially as a single diamond with the vertices labeled as in the left of Figure \ref{fig:boxbybox}.  Note the growth rule starting from $\nu \subset \mu' \subset \lambda$ would recover $\mu$.  In other words, the rules used to pass through a diamond left to right and right to left are the same.

Sch\"utzenberger involution is a map on the set of standard Young tableaux.  Let $T$ be a standard Young tableau and let $\lambda^{(0)} \subset \lambda^{(1)} \subset \ldots \subset \lambda^{(n)}$ be the corresponding chain of shapes.  Place the shapes $\lambda^{(k)}$ along the left of a triangular diagram as in Figure \ref{fig:evac} and place $\emptyset$ at each vertex along the bottom.  The growth rule can then be used from left to right to attach shapes to all the vertices of the diagram.  Let $\Evac(T)$ be the standard Young tableau corresponding to the chain of shapes on the line from the bottom right to the top of the diagram.  Then $\Evac(T)$ is the outcome of Sch\"utzenberger involution (also known as evacuation) applied to $T$.

\begin{figure}
\begin{pspicture}(-.5,-.5)(9,5)
\psline(0,0)(4,4)(8,0)
\psline(5,3)(2,0)(1,1)
\psline(6,2)(4,0)(2,2)
\psline(7,1)(6,0)(3,3)
\uput[ul](0,0){$\emptyset$}
\uput[u](2,0){$\emptyset$}
\uput[u](4,0){$\emptyset$}
\uput[u](6,0){$\emptyset$}
\uput[ur](8,0){$\emptyset$}
\uput[ul](1,1){\begin{shape}{1}{1} \cell{1}{1} \end{shape}}
\uput[u](3,1){\begin{shape}{1}{1} \cell{1}{1} \end{shape}}
\uput[u](5,1){\begin{shape}{1}{1} \cell{1}{1} \end{shape}}
\uput[ur](7,1){\begin{shape}{1}{1} \cell{1}{1} \end{shape}}
\uput[ul](2,2){\begin{shape}{2}{1} \cell{1}{1} \cell{2}{1} \end{shape}}
\uput[u](4,2){\begin{shape}{2}{1} \cell{1}{1} \cell{2}{1} \end{shape}}
\uput[ur](6,2){\begin{shape}{1}{2} \cell{1}{1} \cell{1}{2} \end{shape}}
\uput[ul](3,3){\begin{shape}{3}{1} \cell{1}{1} \cell{2}{1} \cell{3}{1} \end{shape}}
\uput[ur](5,3){\begin{shape}{2}{2} \cell{1}{1} \cell{1}{2} \cell{2}{1} \end{shape}}
\uput[u](4,4){\begin{shape}{3}{2} \cell{1}{1} \cell{1}{2} \cell{2}{1} \cell{3}{1} \end{shape}}
\end{pspicture}
\caption{The growth diagram used to compute Sch\"utzenberger involution, with an example calculation carried out}
\label{fig:evac}
\end{figure}
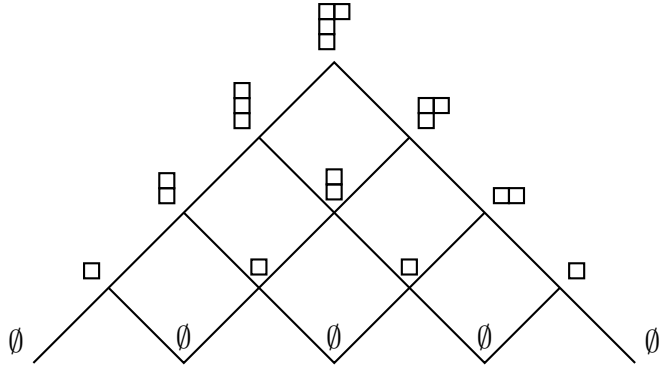

The example given in Figure \ref{fig:evac} illustrates the calculation
\begin{displaymath}
\Evac\left(\ \tableau[sY]{ 1, 4\\ 2\\3}\ \right) =  \tableau[sY]{1, 2\\ 3\\ 4}.
\end{displaymath}
The fact that the growth rule applies the same from right to left as it does from left to right implies that this operation is in fact an involution.

The second operation we define is \emph{jeu-de-taquin}, or jdt for short, which takes as input a skew tableau $T$ and produces a straight tableau $\jdt(T)$ of the same size.  The output is sometimes called the \emph{rectification} of $T$.  Suppose $T$ is of shape $\nu / \mu$.  Fix an arbitrary standard Young tableau $S$ of shape $\mu$.  Consider a rectangular array of diamonds where the bottom left edge consists of the borders of $|\mu|$ diamonds while the top left edge touches $|\nu| - |\mu|$ diamonds.  Figure \ref{fig:jdt} illustrates the case when $|\mu|=3$ and $|\nu|=7$.   Place the chains corresponding to $S$ and $T$ along the bottom left and top left edges respectively.  Apply the growth rules from left to right to obtain a straight tableau $U$ on the bottom right edge and a skew tableau $V$ on the top right.

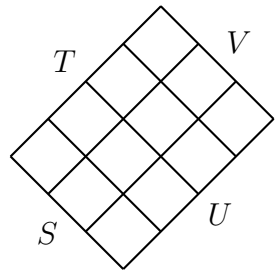
\begin{figure}
\psset{unit=.5cm}
\begin{pspicture}(7,7)
\multirput(3,0)(-1,1){4}{\psline(0,0)(4,4)}
\multirput(3,0)(1,1){5}{\psline(0,0)(-3,3)}
\uput[dl](1.5,1.5){$S$}
\uput[ul](2,5){$T$}
\uput[dr](5,2){$U$}
\uput[ur](5.5,5.5){$V$}
\end{pspicture}
\psset{unit=1cm}
\caption{A growth diagram that computes $\jdt(T)=U$}
\label{fig:jdt}
\end{figure}

\begin{prop}[Confluence of jdt]
The straight tableau $U$ depends only on $T$ and not on $S$.
\end{prop}

In light of the above one can define $\jdt(T) = U$.  Note that each Southwest to Northeast diagonal in the growth diagram encodes a standard Young tableau, the first being $T$ and the last being $U$.  These are precisely the intermediate tableaux that arise when defining jdt in the more standard way as a sequence of inward slides.  The input $S$ encodes the order in which the slides are performed.  Conversely, the portion $V$ of the output encodes the sequence of outward slides that can be applied to $U$ in order to reconstruct $T$.

\subsection{Semistandard Young tableaux} \label{sec:semistandard}
A \emph{semistandard Young tableau} of shape $\lambda$ is a filling of the Young diagram of $\lambda$ that weakly increases along rows and strictly increases down columns.  Any given positive integer is allowed to appear multiple times, or not at all.  Semistandard tableaux of skew shape are defined the same way.

The \emph{Bender-Knuth moves} $t_1,t_2,\ldots$ are certain operations on the set of semistandard Young tableaux of a given shape.  Let $T$ be such a tableau and let $S = T|_{[i,i+1]}$, the skew tableau obtained by taking only the boxes of $T$ with entry equal to $i$ or $i+1$.  Each row of $S$, read left to right, consists of
\begin{enumerate}
\item $a$ entries equal to $i$ that lie directly above an $i+1$,
\item $b$ entries equal to $i$ that are alone in their columns,
\item $c$ entries equal to $i+1$ that are alone in their columns, and finally
\item $d$ entries equal to $i+1$ that lie directly below an $i$
\end{enumerate}
for some $a,b,c,d \geq 0$.  Define a tableau $S'$ by switching the $b$ and $c$ values in each row of $S$ and note the semistandard conditions are still satisfied.  By definition, $t_i(T)$ is the result of replacing $S$ with $S'$ inside of $T$.

As an example, a possible $S$ used in the calculation of $t_3$ of some larger tableau $T$ could be
\begin{displaymath}
S = \tableau[sY]{   \bl & \bl & \bl & \bl & \bl & \bl & 3  & 3 & 4
             \\ \bl & \bl & 3   & 3   & 4   & 4   & 4
							\\ 3   & 3   & 4  
}
\end{displaymath}						
which would transform to 
\begin{displaymath}
S' = \tableau[sY]{   \bl & \bl & \bl & \bl & \bl & \bl & 3  & 3 & 4
             \\ \bl & \bl & 3   & 3   & 3   & 4   & 4
							\\ 4   & 4   & 4  
}.
\end{displaymath}
In the first row, $b=c=1$ so there is no change.  In the second row $b=1$ and $c=2$, so the first $4$ in the row is decremented in $S'$ to obtain $b=2$, $c=1$.  In the last row $b=2$ and $c=0$ so both $3$'s are incremented.

Let $\delta_m$ denote the \emph{staircase shape} partition $\delta_m = (m,m-1,m-2,\ldots, 1)$.  The skew shape $\delta_m / \delta_{m-1}$ consists of $m$ boxes no two of which are in the same row or column, so arbitrary fillings by positive integers result in semistandard Young tableaux of this shape.  Let $W_n$ denote the set of words on the alphabet $\{1,2,\ldots, n\}$.  Given $w=w_1w_2\cdots w_m \in W_n$ of length $m$, let $T(w)$ denote the semistandard Young tableau of shape $\delta_m / \delta_{m-1}$ with $w_1$ in the bottom-left box, $w_2$ in the next box up, and so on.  

The following is a well known connection between jdt and RSK. 

\begin{prop}\cite[Corollay A1.2.6]{St} \label{prop:rskjdt}
The RSK insertion tableau of a word $w$ is the result of jdt on $T(w)$.
\end{prop}

\subsection{Crystal operators}
We will need a few results from crystal theory for some of the proofs; we review these here. First recall the definition of crystal operators. For each $1\leqslant i < n$ we can define $f_i:W_n\to W_n$ as follows. Ignore all the letters besides $i$ and $i+1$. Label every $i$ by a closing parenthesis and every $i+1$ by an open parenthesis. Now match the parentheses in the standard fashion, starting with the innermost ones. Now change the $i$ corresponding to the last unmatched closing parenthesis, if there are any, to an $i+1$ (if there aren't any, do nothing).

\begin{ex} The process for applying $f_2$  to the word $12442313423211243\in W_4$ is illustrated below. The colors show the parenthesis matching, and the boxed parenthesis is the one whose corresponding entry changes. 
$$\begin{array}{ccccccccccccccccc}
1&2&4&4&2&3&1&3&4&2&3&2&1&1&2&4&3\\
 &)& & &\boxed{)}&\textcolor{green}{(}& &\textcolor{red}{(}& &\textcolor{red}{)}&\textcolor{blue}{(}&\textcolor{blue}{)}& & &\textcolor{green}{)}& &(
\end{array}$$
We conclude that $f_2(12442313423211243) = 12443313423211243$.
\end{ex} 

Now we can define (via reading words) the operators $f_i$ on semistandard skew-tableaux of a given shape $\lambda/\mu$. Take a reading word of the tableau (for concreteness, we read each row, starting with the bottom one, from left to right). Applying $f_i$ to the reading word changes a certain $i$ to an $i+1$; we define the value of $f_i$ on the tableau by switching the corresponding $i$ in the tableau to an $i+1$. While not entirely obvious, it is true that the resulting tableau is semistandard.

\begin{ex}
Using our work from the previous example, we conclude that
\[f_2\left(\ 
\tableau[sY]{ \bl & \bl & \bl & \bl & \bl & \bl & \bl & \bl & 3
             \\ \bl & \bl & \bl & \bl & \bl & 1 & 1 & 2 & 4
             \\ \bl & \bl & \bl & \bl & \bl & 2
             \\ \bl & \bl & \bl & \bl & 2 & 3
             \\ \bl & \bl & \bl & 1   & 3 & 4
             \\ \bl & \bl & 2   & 3
             \\ 1   & 2   & 4   & 4}
\ \right) = 
\tableau[sY]{ \bl & \bl & \bl & \bl & \bl & \bl & \bl & \bl & 3
             \\ \bl & \bl & \bl & \bl & \bl & 1 & 1 & 2 & 4
             \\ \bl & \bl & \bl & \bl & \bl & 2
             \\ \bl & \bl & \bl & \bl & 2 & 3
             \\ \bl & \bl & \bl & 1   & 3 & 4
             \\ \bl & \bl & 3   & 3
             \\ 1   & 2   & 4   & 4}.
\]
\end{ex} 

Now we can recall the necessary results about crystal operators.

\begin{prop}[{e.g. \cite[Theorem 3.3.1]{vl}}]
\label{prop:crystal commute with jdt}
Crystal operators commute with jdt slides.
\end{prop}

\begin{prop}[{follows from \cite[Corollary 3.3.3]{vl}}]
\label{prop:crystal connected}
Any two tableaux of the same straight shape are connected by crystal operators.
\end{prop}

\begin{prop}[{\cite[Theorem 5.5.1]{LLT}}]
\label{prop:crystal preserve Q}
Crystal operators on words preserve the RSK recording tableau.
\end{prop}

\begin{prop}[{follows from \cite[Corollary 3.3.3]{vl}}]
\label{prop:haiman-like result}
Suppose $T$ and $T'$ are semistandard Young tableaux of the same shape, that they are connected by crystal operators, and that $\jdt(T)=\jdt(T')$.  Then $T=T'$.   
\end{prop}

 %

\section{Semistandard growth diagrams}
In this section we define a generalized notion of growth diagram that operates on semistandard tableaux rather than just standard ones.  It is convenient to introduce more notation in the standard case before generalizing.

Consider an $(n+1)$-row checkerboard lattice 
\begin{displaymath}
V = \{(i,j) \in \mathbb{Z}^2 : i+j \textrm{ even}, 0 \leq j \leq n\}
\end{displaymath}
with rows numbered $0$, $1$, \ldots, $n$ from bottom to top and extending infinitely to the left and right.  
We will consider certain assignments of a partition $\operatorname{sh}(v)$ to each $v \in V$ satisfying
\begin{enumerate}
\item $\operatorname{sh}(v) = \emptyset$ (the empty shape) for all $v$ on row $0$
\item if $u,v \in V$ with $v = u + (-1,1)$ or $v= u + (1,1)$ then $\operatorname{sh}(v)$ equals $\operatorname{sh}(u)$ with one box added.  In this case say $v$ is a Northwest step or Northeast step, respectively, from $u$.
\end{enumerate}

By a \emph{path} we mean a path consisting entirely of Northwest and Northeast steps.  A path $v_0v_1\cdots v_n$ gives rise to a saturated chain of shapes $\operatorname{sh}(v_0) \subset \operatorname{sh}(v_1) \subset \ldots \subset \operatorname{sh}(v_n)$ which in turn corresponds to a standard Young tableau of shape $\operatorname{sh}(v_n)$.  We write $[v_0\cdots v_n]$ to denote this tableau.  More generally, a path $v_iv_{i+1}\cdots v_j$ from row $i$ to row $j>i$ corresponds to a standard Young tableau of skew shape $\operatorname{sh}(v_j) / \operatorname{sh}(v_i)$ which we denote $[v_i\cdots v_j]$.  Lastly, if $v = u + (-k,k)$ or $v = u + (k,k)$ then write $uv$ for the unique path from $u$ to $v$ and $[uv]$ for the corresponding tableau.

A \emph{primitive diamond} is one whose endpoints $u,v,v',w \in V$ are arranged as $v = u+(-1,1)$, $v' = u + (1,1)$, and $w = v + (1,1) = v'+(-1,1)$.  A \emph{growth diagram} is a (typically connected) finite collection of primitive diamonds.  A growth diagram defines a map on standard Young tableaux as follows.  The input is the tableau $[v_0 \cdots v_n]$ associated to some path entirely left of the diagram.  Shapes are constructed recursively from left to right.  If $v'$ is the right vertex of a primitive diamond belonging to the diagram, then $\operatorname{sh}(v')$ is chosen according to the growth rule.  Otherwise, $\operatorname{sh}(v') = \operatorname{sh}(v)$ where $v = v' - (2,0)$.  The output is the tableau $[v_0' \cdots v_n']$ associated to any path lying entirely to the right of the diagram.

In this language, the growth diagram that computes Sch\"utzenberger involution consists of those diamonds lying right of the path $[ac]$ and left of the path $[bc]$ where $a=(0,0)$, $b=(2i,0)$, and $c = (i,i)$.  Denote this growth diagram $\Evac_i$, so for instance the growth diagram in Figure \ref{fig:evac} is $\Evac_4$.  For any filling of the growth diagram we have $[bc] = \Evac([ac])$.  Note we allow that $i < n$, where $n$ is the height of the lattice.  In this case, $\Evac_i$ acts on standard Young tableaux of size $n$, but by our convention it only acts non-trivially on the part of the tableau with entries $\leq i$.  

A jeu-de-taquin growth diagrams consists of diamonds right of the path $abd$ and left of the path $acd$ where $a=(0,0)$, $b=(-i,i)$, $c=(j,j)$, and $d=(-i+j,i+j)$.  These diamonds form a rectangle, so we denote the diagram $\Rect_{i,j}$.  More precisely, $\Rect_{i,j}$ computes jeu-de-taquin on skew tableaux of shape $\lambda / \mu$ where $|\mu| = i$ and $|\lambda| = i+j$.  As already discussed, any filling of the diagram will have the property that $[ac] = \jdt([bd])$.  The growth diagram $\Rect_{3,4}$ is pictured in Figure \ref{fig:jdt}.

Note that shifting an entire growth diagram left or right does not affect its functionality, so we consider the above diagrams to be defined modulo such shifts.

\subsection{Definition via Bender-Knuth moves} \label{sec:BK}
Let $S$ be a standard Young tableau.  Applied to $S$, the Bender-Knuth move $t_i$ has the effect of interchanging the positions of the $i$ and $i+1$ if possible, and otherwise doing nothing.  Viewing $S$ as a chain of shapes, $\lambda^{(0)} \subset \lambda^{(1)} \subset \ldots \subset \lambda^{(n)}$, only the shape $\lambda^{(i)}$ changes.  Moreover, it changes exactly by the growth rule applied to a diamond starting from shapes $\lambda^{(i-1)}$, $\lambda^{(i)}$, and $\lambda^{(i+1)}$.

In light of the above, it is natural to define semi-standard growth diagrams using Bender-Knuth moves.  As before let $V = \{(i,j) \in \mathbb{Z}^2 : i+j \textrm{ even}, 0 \leq j \leq n\}$.  Assign a partition $\operatorname{sh}(v)$ to each $v \in V$ so that
\begin{enumerate}
\item $\operatorname{sh}(v) = \emptyset$ (the empty shape) for all $v$ on row $0$
\item if $v$ is a Northwest step or Northeast step from $u$ then $\operatorname{sh}(v)$ equals $\operatorname{sh}(u)$ with a (possibly empty) horizontal strip added.
\end{enumerate}
Note the second rule has been modified, so paths now correspond to semistandard Young tableau.  More precisely, a path $v_iv_{i+1}\cdots v_j$ from row $i$ to row $j>i$ gives rise to a semistandard tableau, denoted $[v_iv_{i+1}\cdots v_j]$ of shape $\operatorname{sh}(v_j)/\operatorname{sh}(v_i)$ with entries $m$ occupying the boxes of the horizontal strip $v_{i+m} / v_{i+m-1}$ for $m=1,\ldots, j-i$.

Suppose we are given a primitive diamond with shapes $\lambda$, $\mu$, and $\nu$ at the bottom, left, and top vertices respectively.  Then these shapes encode a semistandard tableau $T$ of skew shape $\nu / \lambda$ with entries all equal to $1$ and $2$.  Another tableau $T'$ of the same shape can be obtained by applying the Bender-Knuth move $t_1$.  The \emph{growth condition} requires that the shape $\mu'$ placed at the right of the diamond is such that the chain $\lambda \subseteq \mu' \subseteq \nu$ encodes $T'$.

Another take on the growth condition is as follows.  Let $v_0v_1\ldots v_n$ be a path and suppose the subpath $v_{i-1}v_iv_{i+1}$ is the left part of a primitive diamond (i.e. it is a Northwest step followed by a Northeast step).  Use the growth rule to determine $\operatorname{sh}(v_i')$ where $v_i'$ is the right vertex of the same diamond.  The effect of replacing $v_i$ with $v_i'$ in the path amounts to the Bender-Knuth move on tableaux
\begin{displaymath}
[v_0 v_1 \cdots v_i' \cdots v_n] = t_i([v_0 v_1 \cdots v_i \cdots v_n]).
\end{displaymath}

Exactly as before, a semistandard growth diagram is a finite collection of primitive diamonds on which the semistandard growth rule is imposed.  Such a diagram encodes a map on semistandard Young tableaux.  In light of the previous paragraph this map is a composition of Bender-Knuth moves $t_j$, one for each primitive diamond with center $(i,j)$, applied from left to right (collections of diamond moves in a common column always commute with each other). As such, a semistandard growth diagram is simply a visual representation of an element of the Berenstein-Kirillov group.

\subsection{Semistandard jeu-de-taquin and Sch\"utzenberger involution}
We can now view the growth diagrams $\Rect_{i,j}$ and $\Evac_j$ as acting on semistandard tableaux.  These actions in turn can be taken to give definitions for jeu-de-taquin and Sch\"utzenberger involution for semistandard tableaux. The next collection of results state that jeu-de-taquin still behaves in the usual fashion.

\begin{lem} \label{lem:semistandard_growth}
Suppose $\lambda, \mu, \mu', \nu$ are partitions at the four corners of a (semistandard) primitive growth diamond (see left of Figure \ref{fig:boxbybox}). Let $\eta\subset\mu$ be the unique partition such that $\mu/\eta$ is the top right cell of the horizontal strip $\mu/\nu$. Consider two primitive growth diamonds as shown on the right side of Figure \ref{fig:boxbybox}. Then $\eta'' = \mu'$ and $\lambda/\mu''$ is the top right cell of the horizontal strip $\lambda/\eta''$.
\end{lem}

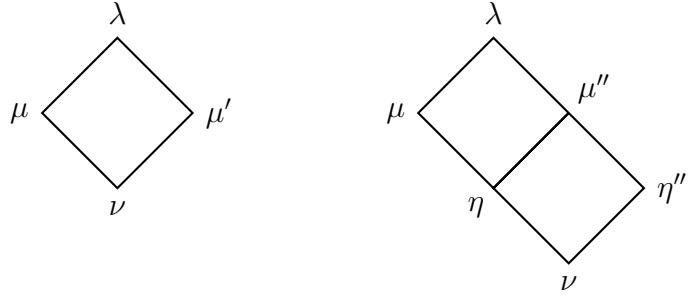
\begin{figure}
\begin{pspicture}(0,-.5)(8,3.5)
\pspolygon(1,1)(0,2)(1,3)(2,2)
\uput[d](1,1){$\nu$}
\uput[l](0,2){$\mu$}
\uput[r](2,2){$\mu'$}
\uput[u](1,3){$\lambda$}

\rput(5,0){
\pspolygon(1,1)(0,2)(1,3)(3,1)(2,0)(1,1)(2,2)
\uput[d](2,0){$\nu$}
\uput[dl](1,1){$\eta$}
\uput[r](3,1){$\eta''$}
\uput[l](0,2){$\mu$}
\uput[ur](2,2){$\mu''$}
\uput[u](1,3){$\lambda$}
}
\end{pspicture}
\caption{The shapes involved in Lemma \ref{lem:semistandard_growth}}
\label{fig:boxbybox}
\end{figure}

\begin{proof}
There are two cases, depending on whether the top right cell of $\mu/\nu$ (whose row we denote by $k$) has a cell of $\lambda$ directly below it. The analysis in the two cases is similar, so we only talk about the case when there is no cell of $\lambda$ directly below the top right cell of $\mu/\nu$. 

The situation is illustrated in Figure \ref{fig:boxbybox_pf}. On the top left $\nu$ is represented by the partition without signs in it, $\mu$ by that partition as well as the cells with the circles, and $\lambda$ by the cells with the circles and the cells with the squares. The bottom left similarly represents the triple $(\nu, \mu', \lambda)$; the step between the top and bottom is a Bender-Knuth move on the circles and squares. Finally, on the right we have the three quadruples $(\nu, \eta, \mu,\lambda)$, $(\nu, \eta, \mu'',\lambda)$, and $(\nu, \eta'', \mu'',\lambda)$. On the right, the first step is a Bender-Knuth move on the triangles and squares while the second step is a Bender-Knuth move on the circles and squares.

The statement that $\mu' = \eta''$ is precisely the statement that in the bottom left and the bottom right, the circles occupy the same positions. This is not difficult to see since the middle picture on the right only differs from the top picture on the left in row $k$, and one easily sees that the number of circles in row $k$ in the bottom picture is unaffected by the difference. The statement that $\lambda/\mu''$ is the top right cell of the horizontal strip $\lambda/\eta''$ means that the triangle in the bottom right picture is no lower than any of the squares. This is clear (using the assumptions that there were no circles above row $k$ in either of the top pictures). 

\begin{figure}
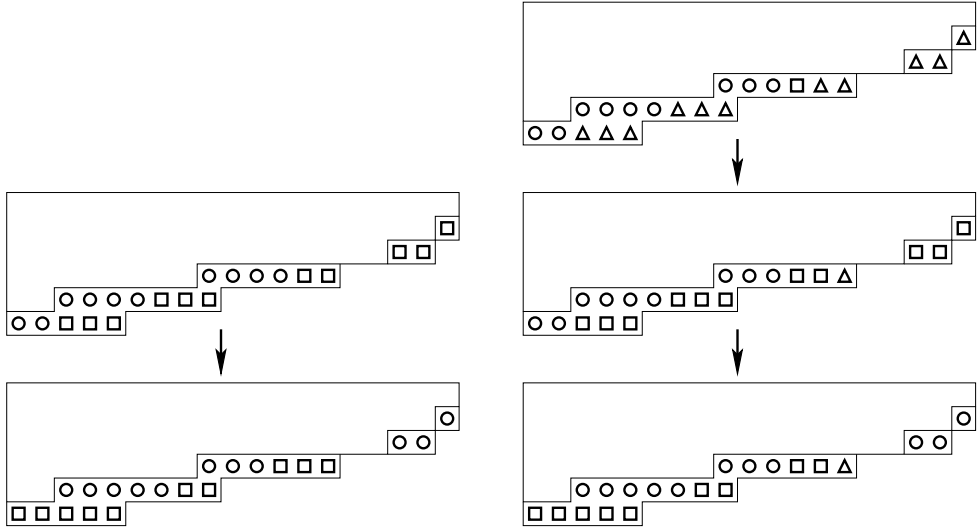

\begin{center}
\resizebox{.4\textwidth}{!}{\input{boxbybox1.pstex_t}}\qquad\resizebox{.4\textwidth}{!}{\input{boxbybox2.pstex_t}}
\end{center}
\caption{Decomposing a semistandard primitive diamond.}
\label{fig:boxbybox_pf}
\end{figure}

\end{proof}

\begin{prop}
Semistandard jeu-de-taquin is confluent. 
\end{prop}
\begin{proof}
Consider a jdt slide on a tableau of shape $\lambda / \mu$ with respect to a horizontal strip $\mu / \nu$; it is the result of applying the growth diagram on the left of Figure \ref{fig:jdt boxbybox}. By repeated applications of the previous lemma, the growth diagram on the right of the same figure (with the shapes on the bottom left side decreasing one box at a time in the appropriate order) produces the same output. Hence a jdt slide with respect to a horizontal strip produces the same result as a sequence of jdt slides with respect to boxes, starting with the top-right box of the horizontal strip and finishing with the bottom-left box of the strip. The fact that jdt with respect to a box is confluent is standard, see \cite{Sch}.
\end{proof}

\begin{figure}
\psset{unit=.5cm}
\begin{pspicture}(0,-.5)(14,7.5)
\rput(0,1){
\multirput(1,0)(1,1){4}{\psline(0,0)(-1,1)}
\multirput(1,0)(-1,1){2}{\psline(0,0)(3,3)}
\uput[d](1,0){$\nu$}
\uput[l](0,1){$\mu$}
\uput[u](3,4){$\lambda$}
}
\rput(7,0){
\multirput(4,0)(1,1){4}{\psline(0,0)(-4,4)}
\multirput(4,0)(-1,1){5}{\psline(0,0)(3,3)}
\uput[d](4,0){$\nu$}
\uput[l](0,4){$\mu$}
\uput[u](3,7){$\lambda$}
}
\end{pspicture}
\psset{unit=1cm}
\caption{Decomposing a jdt slide with respect to a horizontal strip into jdt slides with respect single boxes.}
\label{fig:jdt boxbybox}
\end{figure}
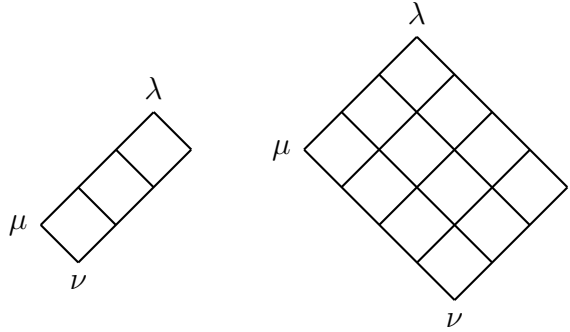


\begin{prop}
As an element of the Berenstein-Kirillov group, the Sch\"utzenberger involution computed by $\Evac_j$ is given by 
\begin{displaymath}
q_{j-1} := t_1(t_2t_1)\cdots(t_{j-1}t_{j-2}\cdots t_1)
\end{displaymath}
\end{prop}
\begin{proof}
The diamonds in $\Evac_j$ have centers 
\begin{displaymath}
(2,1),(4,1),\ldots, (2j-2,1), (3,2),(5,2),\ldots, (2j-3,2), \ldots, (j,j-1).
\end{displaymath}  
Although the recipe says to apply diamond rules left to right, any order is possible provided that that inputs to each step have been computed by previous steps.  Different orderings correspond to different representations of the same element in the BK group.  The desired representation arises by doing moves in the order
\begin{displaymath}
(2,1),(3,2),\ldots, (j,j-1), (4,1),(5,2), \ldots, (j+1,j-2), \ldots , (2j-2,1). 
\end{displaymath}
\end{proof}

\section{Cactus group relations inside the Berenstein-Kirillov group}
\subsection{Definition of the cactus operators}
We use the growth diagrams $\Rect_{i,j}$ and $\Evac_j$ as building blocks to construct more complicated ones.  Each rectangle in subsequent pictures is considered to be a copy of some $\Rect_{i,j}$ growth diagram and each triangle a copy of some $\Evac_j$ growth diagram.  Note these pictures are merely schematic; in reality $\Evac_j$ is not quite a triangle because its bottom boundary is zig-zagged.

Let $\Evac_{[i,j]}$ be the growth diagram given in Figure \ref{fig:q_ij}.  The five component pieces occurring from left to right to make up $\Evac_{[i,j]}$ are $\Evac_{i-1}$, $\Rect_{i-1,j-i+1}$, $\Evac_{j-i+1}$, $\Rect_{j-i+1,i-1}$, and $\Evac_{i-1}$.  Let $q_{[i,j]}$ denote the corresponding map of semistandard tableaux.

\begin{figure}
\begin{pspicture}(-1,-1)(15,5)
\psline(3,3)(6,0)(10,4)(14,0)(0,0)(4,4)(8,0)(11,3)
\uput[ul](1.5,1.5){$i-1$}
\uput[ur](5.5,2.5){$i-1$}
\uput[ul](8.5,2.5){$i-1$}
\uput[ur](12.5,1.5){$i-1$}
\uput[ul](3.5,3.5){$j-i+1$}
\uput[ur](10.5,3.5){$j-i+1$}

\uput[l](0,0){$a$}
\uput[d](6,0){$b$}
\uput[d](8,0){$c$}
\uput[r](14,0){$d$}
\uput[u](7,1){$e$}
\uput[u](4,4){$f$}
\uput[u](10,4){$g$}
\uput[ul](3,3){$u$}
\uput[ur](11,3){$v$}
\end{pspicture}
\caption{The growth diagram $\Evac_{[i,j]}$.}
\label{fig:q_ij}
\end{figure}

\begin{ex}
We will show that 
\[q_{[2,4]}\left(\ \tableau[sY]{1, 2, 4\\2, 3\\ 5}\ \right) = \tableau[sY]{1, 3, 4\\2, 4\\ 5}.\]
The calculation is carried out in Figure \ref{fig:q action example}.

\begin{figure}
\begin{pspicture}(0,-.5)(10,5.5)
\pspolygon(2,0)(1,1)(4,4)(8,0)(9,1)(6,4)
\psline(2,2)(4,0)(7,3)
\psline(3,3)(6,0)(8,2)
\psdots(0,0)(10,0)(5,5)
\uput[d](0,0){$\emptyset$}
\uput[d](2,0){$\emptyset$}
\uput[d](4,0){$\emptyset$}
\uput[d](6,0){$\emptyset$}
\uput[d](8,0){$\emptyset$}
\uput[d](10,0){$\emptyset$}
\uput[l](1,1){$1$}
\uput[l](3,1){$2$}
\uput[l](5,1){$1$}
\uput[l](7,1){$1$}
\uput[r](9,1){$1$}
\uput[l](2,2){$21$}
\uput[l](4,2){$21$}
\uput[l](6,2){$2$}
\uput[r](8,2){$11$}
\uput[l](3,3){$22$}
\uput[l](5,3){$31$}
\uput[r](7,3){$21$}
\uput[l](4,4){$32$}
\uput[r](6,4){$32$}
\uput[l](5,5){$321$}
\end{pspicture}
\caption{The action of $q_{[2,4]}$.}
\label{fig:q action example}
\end{figure}
\end{ex}

\begin{prop} \label{prop:qij}
\begin{displaymath}
q_{[i,j]} = q_{j-1}q_{j-i}q_{j-1}
\end{displaymath}
\end{prop}

\begin{proof}
Both sides clearly only operate on entries of tableaux between $1$ and $j$ inclusive, so we assume our input $T$ to $q_{[i,j]}$ has all entries at most $j$.  Using the vertex labeling in Figure \ref{fig:q_ij} we have $[af] = T$.  Viewing triangle $afc$ as a single copy of $\Evac_j$ we have $[cf] = q_{j-1}(T)$.  Then, running the growth diagram represented by triangle $bec$ backwards yields that $[bef] = q_{j-i}(q_{j-1}(T))$.  By our conventions, $[ef]=[eg]$ as all of the intervening diamonds lie outside the growth diagram, so $[beg] = [bef] = q_{j-i}(q_{j-1}(T))$.  Finally, viewing triangle $bgd$ as a copy of $\Evac_j$ yields
\begin{displaymath}
[dg] = q_{j-1}(q_{j-i}(q_{j-1}(T))).
\end{displaymath}
\end{proof}

\subsection{Locality}
\begin{prop} \label{prop:qij_local}
Let $T$ be a semistandard Young tableau with maximum entry at most $n$ and let $T' = q_{[i,j]}(T)$ for some $1 \leq i < j \leq n$.  Then $T'$ possesses, and is uniquely determined by, the following properties:
\begin{enumerate}
\item $T'|_{[1,i-1]} = T|_{[1,i-1]}$ and $T'|_{[j+1,n]} = T|_{[j+1,n]}$,
\item $P = \jdt(T|_{[i,j]})$ and $P' = \jdt(T'|_{[i,j]})$ are related by evacuation,
\item $T_{[i,j]}$ and $T'_{[i,j]}$ are related by a sequence of crystal operators. 
\end{enumerate}
\end{prop}

\begin{proof} \ 
\begin{enumerate}
\item Using the vertex labeling from Figure \ref{fig:q_ij} we have
\begin{displaymath}
T_{[1,i-1]} = [au] = \Evac([bu]) = \Evac(\jdt([ef]))
\end{displaymath}
and
\begin{displaymath}
T_{[1,i-1]} = [dv] = \Evac([cv]) = \Evac(\jdt([eg]))
\end{displaymath}
which are equal since $ef=eg$.  The diagram does not extend above height $j$, so $T_{[j+1,n]} = T'_{[j+1,n]}$ is immediate.
\item In Figure \ref{fig:q_ij}, $P = [be]$ and $P' = [ce]$ which are the input and output of a growth diagram for evacuation.
\item By Proposition \ref{prop:crystal connected}, $P$ and $P'$ are connected by crystal operators.  Now $[ef]=[eg]$ so $T_{[i,j]}$ and $T'_{[i,j]}$ are obtained from $P$ and $P'$ respectively by a common sequence of outward jdt slides.  It follows from Proposition \ref{prop:crystal commute with jdt} that $T_{[i,j]}$ and $T'_{[i,j]}$ are related by crystal operators.
\end{enumerate}
The uniqueness follows from Proposition \ref{prop:haiman-like result}.
\end{proof}

Proposition \ref{prop:qij_local} in particular implies that $q_{[i,j]}$ acts locally in the following sense.  If $T' = q_{[i,j]}(T)$ then the first statement implies that $T$ and $T'$ agree except with regards to the placement of the entries from $i$ to $j$.  Moreover, the second and third statements characterize $T'|_{[i,j]}$ in such a way to make it clear that it depends only on $T_{[i,j]}$ and not on the rest of $T$.  The second cactus relation, that $q_{[i,j]}$ and $q_{[k,l]}$ commute when $[i,j]$ and $[k,l]$ are disjoint follows immediately.

\subsection{Proof of the cactus relations}
We are ready to prove Theorem \ref{thm:map} which we reformulate as follows.
\begin{theorem*} The maps $q_{[i,j]}$ for $1 \leq i < j \leq n$ satisfy the relations of the cactus group, namely
\begin{itemize}
\item if $i < j$ then $q_{[i,j]}^2 = I$,
\item if $i < j < k < l$ then $q_{[i,j]}q_{[k,l]} = q_{[k,l]}q_{[i,j]}$,
\item if $i \leq j < k \leq l$ then $q_{[i,l]}q_{[j,k]}q_{[i,l]} = q_{[j',k']}$ where $j' = i+l-k$ and $k' = i+l-j$.
\end{itemize}
\end{theorem*}

\begin{proof}
The first relation follows from Proposition \ref{prop:qij} and the fact that evacuation is an involution.  The second relation is clear since $q_{[i,j]}$ and $q_{[k,l]}$ operate on disjoint parts of tableaux as established in Proposition \ref{prop:qij_local}.  We claim for the third relation that it suffices to prove the $i=1$ case.  Indeed, assuming this case and given $i \leq j < k \leq l$ we have
\begin{align*}
q_{[i,l]}q_{[j,k]}q_{[i,l]} &= (q_{[1,l]}q_{[1,l-i+1]}q_{[1,l]})q_{[j,k]}(q_{[1,l]}q_{[1,l-i+1]}q_{[1,l]}) \quad \textrm{ (by Proposition \ref{prop:qij})} \\
&= q_{[1,l]}q_{[1,l-i+1]}q_{[1+l-k,1+l-j]}q_{[1,l-i+1]}q_{[1,l]} \\ 
&= q_{[1,l]}q_{[j-i+1,k-i+1]}q_{[1,l]} \\ 
&= q_{[i+l-k,i+l-j]} 
\end{align*}
where the last three steps are all applications of the $i=1$ case of the relation.

It remains to prove the $i=1$ case of the third relation.  Suppose $1 \leq j < k \leq l$.  Let a tableau $T$  be given.  Let
\begin{displaymath}
T' = q_{[1,l]}(q_{[j,k]}(q_{[1,l]}(T))).
\end{displaymath}
Apply the $\Evac_{l-1}$ growth diagram twice, once with input $T$ and once with input $T'$.  Recall that $\Evac_{l-1}$ computes $q_{l-1} = q_{[1,l]}$ so the outputs are $q_{l-1}(T)$ and $q_{l-1}(T') = q_{[j,k]}(q_{l-1}(T))$.  Break the growth diagram with input $T$ into six regions and label vertices as in Figure \ref{fig:cactus}.  Perform the same subdivision of the growth diagram with input $T'$, calling corresponding vertices $a', b', \ldots$.

We claim that the three regions labeled $*$ are filled in an identical manner in the two growth diagrams, whereas the fillings of the region labeled $\longleftrightarrow$ are reflections of each other.  This verification takes several steps, which we number for convenience.

\begin{enumerate}
\item The outputs $[dv] = q_{l-1}(T)$ and $[d'v'] = q_{[j,k]}(q_{l-1}(T))$ are related by $q_{[j,k]}$ and hence only differ in the interval $[j,k]$.  Therefore $[dh] = [d'h']$ and $[uv] = [u'v']$.
\item Also by definition of $q_{[j,k]}$ we have that $[ce] = \jdt([hu])$ and $[c'e'] = \jdt([h'u'])$ are related by evacuation.  
\item Applying the growth rules from right to left starting from $[dh]=[d'h']$ yields $[ch]=[c'h']$.  
\item Consider a Southeast to Northwest diagonal from edge $ch$ to edge $eu$.  Such encodes a skew tableau obtained by performing some $\jdt$ slides to $[hu]$ in one diagram and $[h'u']$ in the other.  Since $[hu]$ and $[h'u']$ are related by crystal operators, so are these two new tableaux, and in particular they have the same shape.  Letting this diagonal vary, we have $[eu] = [e'u']$.
\item As $[eu] = [e'u']$ and $[uv] = [u'v']$, the growth rules imply $[eg] = [e'g']$ and $[gv] = [g'v']$.
\item Triangle $bec$ is a small evacuation growth diagram, so $[b'e'] = \Evac([c'e']) = [ce] = \Evac([be])$.
\item Applying $\jdt$ to both sides of $[eg] = [e'g']$ yields $[bf] = [b'f']$.
\item The growth rules starting from $[bf] = [b'f']$ yields $[af]=[a'f']$.
\end{enumerate}

As $[af] = [a'f']$ and $[gv]=[g'v']$, the inputs $T=[av]$ and $T' = [a'v']$ agree outside the interval $[l-k+1,l-j+1]$.  Now consider the skew tableaux $[fg]$ and $[f'g']$ obtained by  restricting to this interval where $T$ and $T'$ differ.   Applying $\jdt$ to each yields $[be]$ and $[b'e']$ which are related by evacuation.  Moreover, $[eg]=[e'g']$ so a common sequence of outward jdt steps recover $[fg]$ and $[f'g']$.  By the characterization in Proposition \ref{prop:qij_local}, $T' = q_{[l-k+1,l-j+1]}(T)$ as desired.
\end{proof}

\begin{figure}
\begin{pspicture}(0,-.5)(10,5)
\pspolygon(0,0)(5,5)(10,0)
\psline(1.5,1.5)(3,0)(6.5,3.5)
\psline(2.5,2.5)(5,0)(7.5,2.5)
\uput[ul](.75,.75){$l-k$}
\uput[ul](2,2){$k-j+1$}
\uput[ul](3.75,3.75){$j-1$}
\uput[ur](8.75,1.25){$j-1$}
\uput[ur](7,3){$k-j+1$}
\uput[ur](5.75,4.25){$l-k$}
\uput[d](0,0){$a$}
\uput[d](3,0){$b$}
\uput[d](5,0){$c$}
\uput[d](10,0){$d$}
\uput[u](4,1){$e$}
\uput[r](1.5,1.5){$f$}
\uput[r](2.5,2.5){$g$}
\uput[l](7.5,2.5){$h$}
\uput[l](6.5,3.5){$u$}
\uput[d](5,5){$v$}
\rput(1.5,.5){$*$}
\rput(7.5,1){$*$}
\rput(4.5,3){$*$}
\rput(4,.3){$\longleftrightarrow$}

\end{pspicture}

\caption{Illustration of the cactus relation $q_{[1,l]}q_{[j,k]}q_{[1,l]} = q_{[l-k+1,l-j+1]}$}
\label{fig:cactus}
\end{figure}
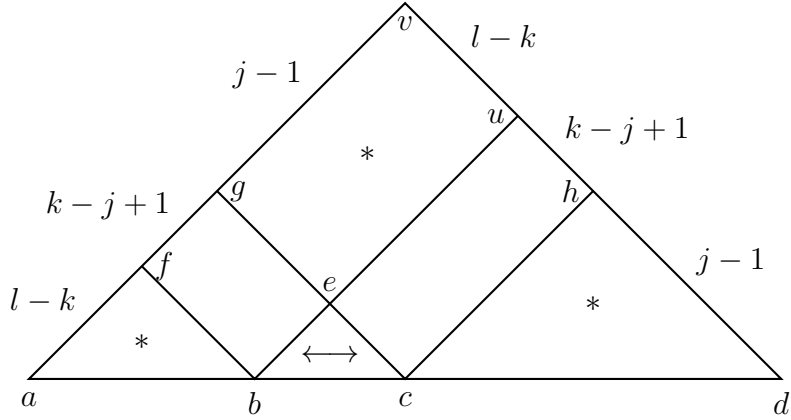

Note each cactus relation can be expanded in terms of the Bender-Knuth moves $t_i$.  The commutation relations are of particular interest because they do not obviously follow from known relations in the BK-group.  Our Theorem \ref{thm:new_relations} adds a family of relations that together with the previously known ones are sufficient to derive commutativity.

\proof[Proof of Theorem \ref{thm:new_relations}.]
By the second cactus relation, $q_i = q_{[1,i+1]}$ and $q_{[j,k]}$ commute whenever $i+1 < j < k$.  We will show that $t_i$ can be expressed in terms of $q_1,\ldots, q_i$, so $t_i$ also commutes with $q_{[j,k]}$.  As both are involutions this will imply $(t_iq_{[j,k]})^2 = I$.

Note that $q_1 = t_1$ and $q_2 = t_1t_2t_1$ so $t_1 = q_1$ and $t_2 = t_1^2t_2t_1^2 = q_1q_2q_1$.  In general $q_{i} = p_1p_2 \cdots p_{i} = q_{i-1}p_{i}$ where $p_j = t_jt_{j-1}\cdots t_1$.  As $q_{i-1}$ is an involution we get $q_{i-1}q_i = p_i$.  Therefore $t_i = p_ip_{i-1}^{-1} = q_{i-1}q_i(q_{i-2}q_{i-1})^{-1} = q_{i-1}q_iq_{i-1}q_{i-2}$ as desired.

\section{Cactus group action on words} \label{sec:words}
Define a growth diagram $\Boom_{[i,j]}$ as in Figure \ref{fig:tau_ij}.  Note the diagram also depends on a positive integer $m$, but we suppress this dependence.  If one attached to this diagram two copies of $\Evac_{m+i-2}$, one to the beginning and one to the end, the result would be $\Evac_{[m-1+i,m-1+j]}$.  This ``conjugation of growth diagrams'' does not affect the locality given by Proposition \ref{prop:qij_local}.  So, in the notation of Figure \ref{fig:tau_ij}, we have that $[eg]$ is a function of $[df]$ alone, and that the rest of the input and output agree.

\begin{figure}
\begin{pspicture}(0,-.5)(12,6.5)
\pspolygon(4,0)(0,4)(2,6)(8,0)(12,4)(10,6)(4,0)(8,0)
\uput[ul](1,5){$j-i+1$}
\uput[ur](11,5){$j-i+1$}
\uput[dl](2,2){$m+i-2$}
\uput[dl](4,4){$m+i-2$}
\uput[dr](8,4){$m+i-2$}
\uput[dr](10,2){$m+i-2$}

\uput[dl](4,0){$a$}
\uput[dr](8,0){$b$}
\uput[u](6,2){$c$}
\uput[l](0,4){$d$}
\uput[r](12,4){$e$}
\uput[u](2,6){$f$}
\uput[u](10,6){$g$}

\end{pspicture}
\caption{The growth diagram $\Boom_{[i,j]}$}
\label{fig:tau_ij}  
\end{figure}
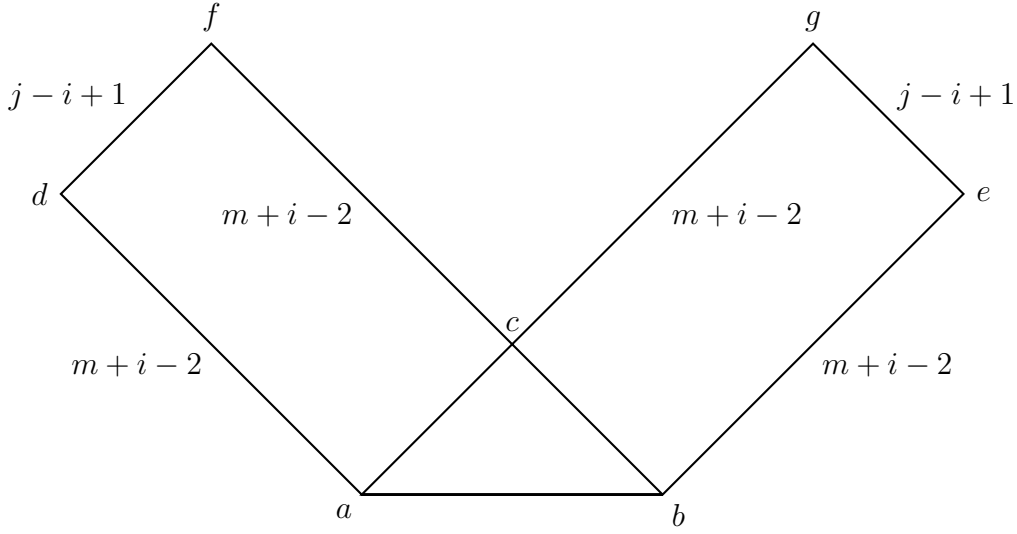 

In what follows, for $S$ a semistandard Young tableau we take the expression $S|_{[i,j]}$ to be the skew tableau obtained by 
\begin{enumerate}
\item taking the subtableau of $S$ consisting of all boxes with entries in the interval from $i$ to $j$ and then
\item decreasing all entries by $i-1$ to facilitate comparing to other tableaux.
\end{enumerate}

We will consider $\Boom_{[i,j]}$ as acting on words of length $m$.  More precisely, let $w \in W_n$ be of length $m$.  Then $T(w)$, as defined in Section \ref{sec:semistandard}, is of shape $\delta_m / \delta_{m-1}$ and has entries between $1$ and $n$. 	Take any $S \in \SSYT(\delta_{m})$ with the property that $S|_{[m,m-1+n]} = T(w)$ as the input to $\Boom_{[i,j]}$.  Let $S'$ be the output.  By locality, $S$ and $S'$ only differ on the interval from $m-1+i$ to $m-1+j$.  In particular, $S'|_{[m,m-1+n]}$ also has shape $\delta_m / \delta_{m-1}$ and hence equals $T(w')$ for some word $w'$.  Also by locality, $w'$ depends only on $w$ and not on the remainder of $S$.  As such, we define $\tau_{[i,j]}: W_n \to W_n$ by $\tau_{[i,j]}(w) = w'$.  

There is even a little more to conclude from locality, namely that $w$ and $w'$ only differ in the placement of the letters in the interval from $i$ to $j$.  To determine the relevant part of $w'$, which is encoded by $T(w')|_{[i,j]} = [eg]$, the input $S$ need only satisfy $S|_{[m-1+i,m-1+j]} = T(w)|_{[i,j]}$, a fact that will be useful shortly.

\begin{ex}
We will show that 
\[\tau_{[2,4]}(215324) = 215434.\]
The calculation is given in Figure \ref{fig:tau action example}. We choose $S$ by filling $\delta_{m-1}$ by anti-diagonals, i.e. with one $1$, two $2$'s, three $3$'s, etc.  Note the locality, namely that the input and output words agree except in the placement of the $2$'s, $3$'s, and $4$'s.
\begin{figure}
\scalebox{.8}{
\begin{pspicture}(-1,-.5)(18,10)
\pspolygon(0,6)(6,0)(15,9)(18,6)(12,0)(3,9)
\psline(1,7)(8,0)(16,8)
\psline(2,8)(10,0)(17,7)
\psdots(4,10)(14,10)
\multirput(5,1)(-1,1){5}{\psline(0,0)(3,3)}
\multirput(13,1)(1,1){5}{\psline(0,0)(-3,3)}
\uput[d](6,0){$\emptyset$}
\uput[d](8,0){$\emptyset$}
\uput[d](10,0){$\emptyset$}
\uput[d](12,0){$\emptyset$}
\uput[l](5,1){$1$}
\uput[l](7,1){$2$}
\uput[l](9,1){$1$}
\uput[l](11,1){$1$}
\uput[l](13,1){$1$}
\uput[l](4,2){$21$}
\uput[l](6,2){$21$}
\uput[l](8,2){$21$}
\uput[l](10,2){$2$}
\uput[l](12,2){$11$}
\uput[l](14,2){$21$}
\uput[l](3,3){$321$}
\uput[l](5,3){$311$}
\uput[l](7,3){$22$}
\uput[l](9,3){$31$}
\uput[l](11,3){$21$}
\uput[l](13,3){$211$}
\uput[l](15,3){$321$}
\uput[l](2,4){$4321$}
\uput[l](4,4){$4211$}
\uput[l](6,4){$321$}
\uput[l](8,4){$32$}
\uput[l](10,4){$32$}
\uput[l](12,4){$311$}
\uput[l](14,4){$3211$}
\uput[l](16,4){$4321$}
\uput[l](1,5){$54321$}
\uput[l](3,5){$53211$}
\uput[l](5,5){$4311$}
\uput[l](7,5){$421$}
\uput[l](11,5){$421$}
\uput[l](13,5){$4211$}
\uput[l](15,5){$43211$}
\uput[l](17,5){$54321$}
\uput[l](0,6){$54322$}
\uput[l](2,6){$553211$}
\uput[l](4,6){$54211$}
\uput[l](6,6){$5311$}
\uput[l](12,6){$5311$}
\uput[l](14,6){$53211$}
\uput[l](16,6){$543211$}
\uput[l](18,6){$54322$}
\uput[l](1,7){$553221$}
\uput[l](3,7){$554211$}
\uput[l](5,7){$64211$}
\uput[l](13,7){$64211$}
\uput[l](15,7){$553211$}
\uput[l](17,7){$543221$}
\uput[l](2,8){$554221$}
\uput[l](4,8){$654211$}
\uput[l](14,8){$654211$}
\uput[l](16,8){$553221$}
\uput[l](3,9){$654221$}
\uput[l](15,9){$654221$}
\uput[l](4,10){$654321$}
\uput[l](14,10){$654321$}
\end{pspicture}
}
\caption{The action of $\tau_{[2,4]}$.}
\label{fig:tau action example}
\end{figure}
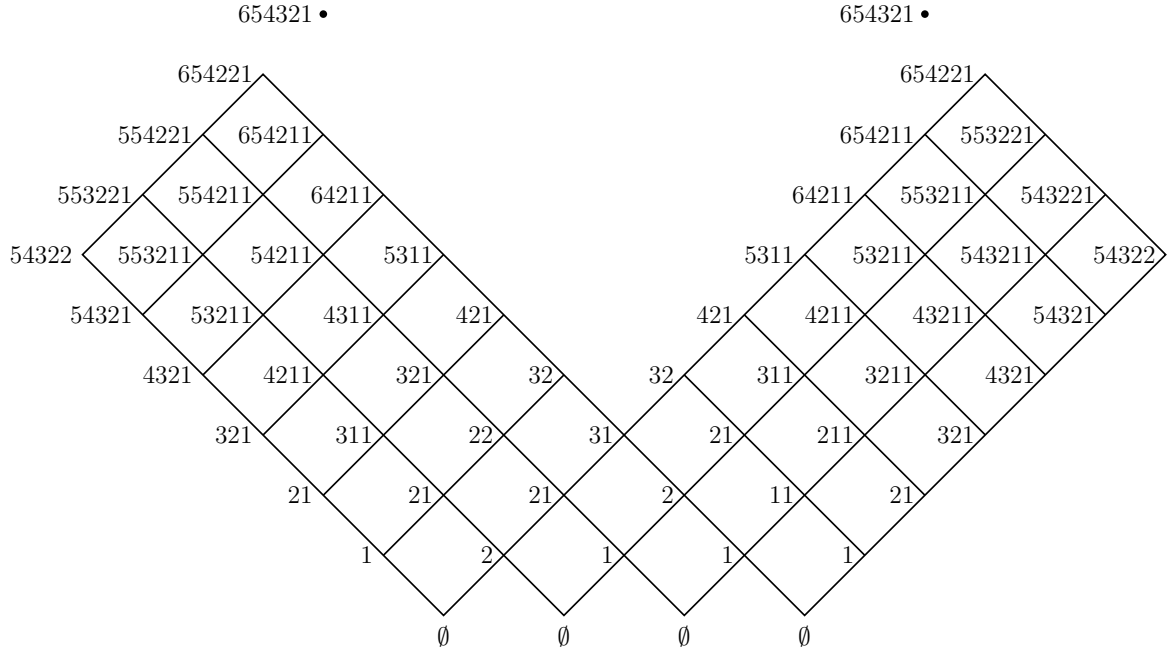
\end{ex}

We now relate the action of $\tau_{[i,j]}$ on words with the action of $q_{[i,j]}$ on tableaux, which will then allow us to conclude that the former satisfies the cactus relations.

\begin{thm} \label{thm:q2tau}
Let $w \in W_n$ of length $m$. Choose a semistandard tableau $S$ of shape $\delta_m$ with $S|_{[m,m-1+n]} = T(w)$.  Let $P = \jdt(T(w))$ (use the sequence of slides dictated by $S|_{[1,m-1]}$) , let $P' = q_{[i,j]}(P)$, and construct $T'$ from $P'$ by undoing the $\jdt$ procedure with the same sequence of reverse slides.  Then $T' = T(w')$ where $w' = \tau_{[i,j]}(w)$.
\end{thm}

\begin{proof}
The procedure described in the statement to go from $T(w)$ to $T'$ amounts to applying the growth diagram in Figure \ref{fig:q2tau} to the tableau $S$. Note that by construction $[lr] = [os]$ and $[rv] = [sw]$, so $[tv] = [uw]$. We will now show that $[hp] = [kq]$. First notice that $[fl] = [fo]$, and by the above argument $[lt] = [ou]$, so $[ft] = [fu]$. Thus $[bp] = \jdt([ft]) = \jdt([fu]) = [cq]$. In particular, $[ep] = [gq]$ and $[ae] = \Evac([be]) = \Evac([cg]) = [dg]$. Thus indeed $[hp] = [kq]$. Finally, it is clear that $[pt]$ and $[qu]$ are related by the action of a $\Boom_{[i,j]}$ diagram for some special input tableau.  Putting this all together yields $T' = [kw] = T(\tau_{[i,j]}(w))$ as desired.

\begin{figure}
\scalebox{.8}{
\begin{pspicture}(-2,-1)(16,8)
\psline(3,3)(6,0)(10,4)(14,0)(0,0)(4,4)(8,0)(11,3)
\uput[l](0,0){$a$}
\uput[d](6,0){$b$}
\uput[d](8,0){$c$}
\uput[r](14,0){$d$}
\uput[u](3,3){$e$}
\uput[u](7,1){$f$}
\uput[u](11,3){$g$}
\uput[u](-2,2){$h$}
\uput[u](16,2){$k$}
\uput[u](4,4){$l$}
\uput[u](10,4){$o$}
\uput[l](1,5){$p$}
\uput[u](13,5){$q$}
\uput[u](6,6){$r$}
\uput[u](8,6){$s$}
\uput[u](2,6){$t$}
\uput[u](12,6){$u$}
\uput[u](4,8){$v$}
\uput[u](10,8){$w$}

\uput[l](-1,1){$m-1$}
\uput[l](-.5,3.5){$i-1$}
\uput[l](1.5,5.5){$j-i+1$}
\uput[l](3,7){$n-j$}

\pspolygon(0,0)(6,6)(4,8)(-2,2)
\psline(3,3)(1,5)
\psline(4,4)(2,6)

\pspolygon(14,0)(8,6)(10,8)(16,2)
\psline(11,3)(13,5)
\psline(10,4)(12,6)
\end{pspicture}
}
\caption{The growth diagram relating $q_{[i,j]}$ and $\tau_{[i,j]}$.}
\label{fig:q2tau}
\end{figure}
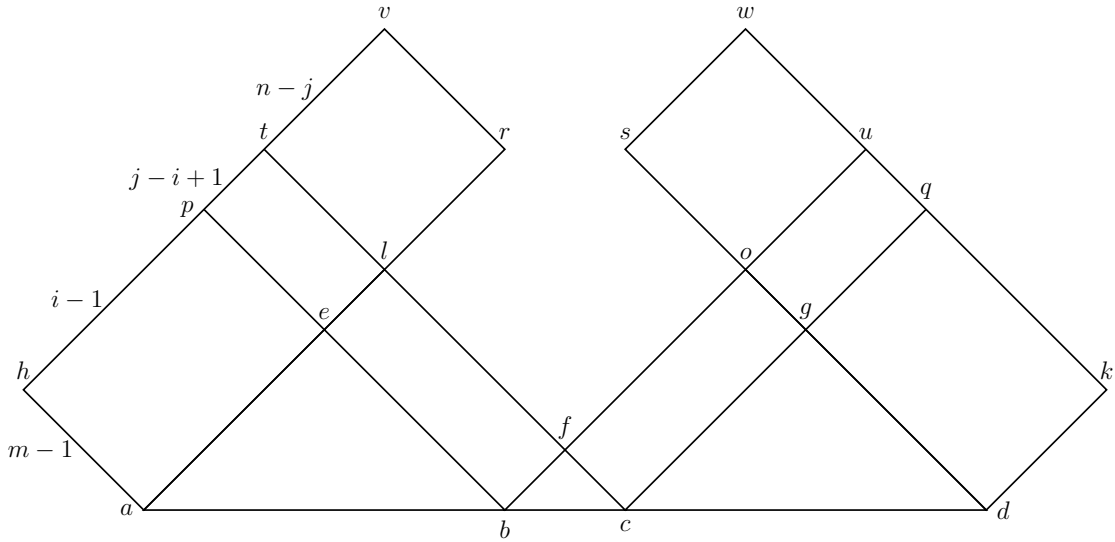
\end{proof}

\begin{rmk}
In particular, the procedure described in the above theorem is independent of the tableau used to do and undo the $\jdt$ procedure.
\end{rmk}


\begin{thm} \label{thm:tauCactus}
The maps $\tau_{[i,j]}$ for $1 \leq i < j \leq n$ satisfy the relations of the cactus group, namely
\begin{itemize}
\item if $i < j$ then $\tau_{[i,j]}^2 = I$,
\item if $i < j < k < l$ then $\tau_{[i,j]}\tau_{[k,l]} = \tau_{[k,l]}\tau_{[i,j]}$,
\item if $i \leq j < k \leq l$ then $\tau_{[i,l]}\tau_{[j,k]}\tau_{[i,l]} = \tau_{[j',k']}$ where $j' = i+l-k$ and $k' = i+l-j$.
\end{itemize}
\end{thm}
\begin{proof}
The first relation is clear since applying the square of $\tau_{[i,j]}$ amounts to applying the growth diagram $\Boom_{[i,j]}$ from left to right and then applying it from right to left. The second relation follows since $\tau_{[i,j]}$ and $\tau_{[k,l]}$ operate on disjoint parts of the word.

Suppose $\tau_{[i,l]}\tau_{[j,k]}\tau_{[i,l]}(w) = w_1$ and $\tau_{[j',k']}(w) = w_2$. By Theorem \ref{thm:q2tau}, $w_1$ is formed by taking $\jdt(T(w))$, applying $q_{[i,l]}q_{[j,k]}q_{[i,l]}$ to it, and then undoing the $\jdt$ with the same upper side. Similarly, $w_2$ is formed by taking $\jdt(T(w))$, applying $q_{[j',k']}$ to it, and then undoing the $\jdt$ with the same upper side. The result follows from the corresponding cactus relation for $q$'s. 
\end{proof}

We now explain how the definition of $\tau_{[i,j]}$ given in this section relates to the one acting on permutations from the introduction.  The latter, in short, applies Sch\"utzenberger involution to the $Q$ tableau of the sub-permutation occupying positions $i$ through $j$, while leaving its $P$ tableau unchanged.  Taking inverses of permutations has the effect of switching the roles of positions and values and also interchanges $P$ and $Q$.  As such, an equivalent theory is obtained by having $\tau_{[i,j]}$ act on the sub-permutation of values $i$ through $j$ by modifying its $P$ tableau.

Now let $w \in W_n$ and $w' = \tau_{[i,j]}(w)$ as defined in the current section.  Let $v$ be the subword of $w$ consisting only of its letters in the interval $i$ through $j$, and similarly for $v'$ with respect to $w'$.  The computation of $w'$ is done using a $\Boom_{[i,j]}$ growth diagram as in Figure \ref{fig:tau_ij}.  A common sequence of jdt moves can be used to bring $[df]$ down to $T(v)$ and $[eg]$ down to $T(v')$.  It follows from Proposition \ref{prop:rskjdt} that $[ac] = P(v)$ and $[bc]=P(v')$, so $P(v') = \Evac(P(v))$.  Meanwhile, $[df]$ is related by crystal operators to $[eg]$, so the same is true of $v$ and $v'$.  By Proposition \ref{prop:crystal preserve Q}, we have $Q(v)=Q(v')$.  As such, the whole procedure is a natural generalization of the action of $\tau_{[i,j]}$ on permutations.

Now that the action of $\tau_{[i,j]}$ on words is explained, Theorem \ref{thm:tauCactus} serves as proof of Corollary \ref{cor:tau}.

\section{Equivalences between cactus-type and BK-type relations}

Consider the free group generated by $t_i$, $i \in \mathbb Z_{>0}$. Consider another free group generated by $q_{[i,j]}$, $1 \leq i < j$. Consider the morphisms $\phi$ and $\psi$ between those groups; $\phi$ is given by  
$$q_{[i,j]} \mapsto q_{j-1} q_{j-i} q_{j-1}, \text{ where } q_i = t_1 t_2 t_1 \dotsc t_i t_{i-1} \dotsc t_1,$$ 
and $\psi$ is given by
$$t_1 \mapsto q_{[1,2]}, \;\; t_2 \mapsto q_{[1,2]}q_{[1,3]}q_{[1,2]}, \;\; t_i \mapsto q_{[1,i]}q_{[1,i+1]}q_{[1,i]}q_{[1,i-1]} \text{ for } i>2.$$ 

In the introduction we stated the following Theorem \ref{thm:rel1}.

\begin{theorem*}
Let 
$$G_1 := \left\langle t_i| t_i^2=I, t_i t_j = t_j t_i \text{ if } |i-j|>1\right\rangle$$
and let 
$$G_2 := \left\langle q_{[i,j]}| q_{[i,j]}^2 = I, q_{[i,j]} q_{[k,l]} q_{[i,j]} = q_{[i+j-l,i+j-k]} \text{ if } i \leq k < l \leq j
	     \right\rangle$$
Then $G_1\cong G_2$, $\phi$ descends to an isomorphism, and $\psi$ to its inverse.
\end{theorem*}

We are now ready to prove this. The proof that the maps $\phi$ and $\psi$ descend to $G_1$ and $G_2$ will be split into four parts. 

\begin{proof} [Proof that $\psi(q_{[i,j]}^2) = I$.]

First, it suffices to show that $q_i^2 = I$, since then $$\psi(q_{[i,j]}^2) = q_{j-1} (q_{j-i} (q_{j-1} q_{j-1}) q_{j-i}) q_{j-1} = I.$$

This can be shown by direct manipulation. A more elegant way however is through the use of a wiring diagram. It is well known that wiring diagrams are in bijection with commutation classes of elements of the symmetric groups, see  \cite{E, MS}. The latter are generated by the Coxeter generators $s_i$ subject to the same relations as our $t_i$-s, plus the braid relations. Thus, if we show that $q_i$ and $q_i^{-1}$ are represented by the same wiring diagram, they must be equal in the symmetric group without relying on the braid relations, and thus they are equal in our group generated by the $t_i$.
 
The proof is then completed by meditation upon Figure \ref{fig:involution}. 
\end{proof}

\begin{figure}
\psset{unit=.7cm}
\begin{pspicture}(20,5)
\multirput(0,0)(11,0){2}{
	\psline(0,1)(1,1)(5,5)(9,5)
	\psline(0,2)(1,2)(2,1)(3,1)(6,4)(9,4)
	\psline(0,3)(2,3)(4,1)(5,1)(7,3)(9,3)
	\psline(0,4)(3,4)(6,1)(7,1)(8,2)(9,2)
	\psline(0,5)(4,5)(8,1)(9,1)}
\rput{45}(3,3){\psellipse[linestyle=dashed](0,0)(3.5,.3)}
\rput{45}(4.5,2.5){\psellipse[linestyle=dashed](0,0)(3,.3)}
\rput{45}(6,2){\psellipse[linestyle=dashed](0,0)(2,.3)}
\rput{45}(7.5,1.5){\psellipse[linestyle=dashed](0,0)(1.5,.3)}
\rput(11,0){
	\rput{135}(6,3){\psellipse[linestyle=dashed](0,0)(3.5,.3)}
	\rput{135}(4.5,2.5){\psellipse[linestyle=dashed](0,0)(3,.3)}
	\rput{135}(3,2){\psellipse[linestyle=dashed](0,0)(2,.3)}
	\rput{135}(1.5,1.5){\psellipse[linestyle=dashed](0,0)(1.5,.3)}
}

\end{pspicture}
\psset{unit=1cm}
\caption{Illustration of the proof of $\psi(q_i) = \psi(q_i^{-1})$ with $i=4$}
\label{fig:involution}
\end{figure}
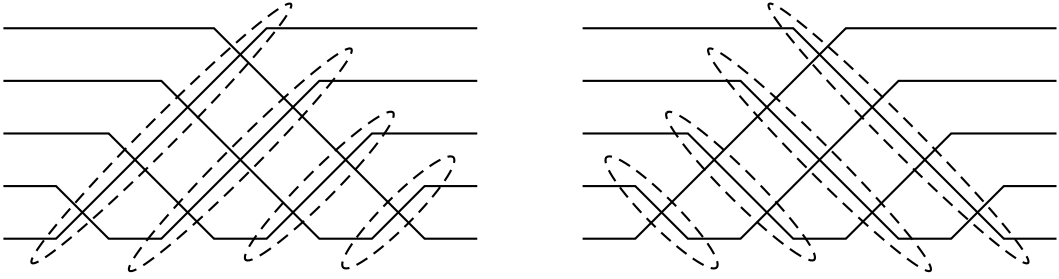

The calculation showing that $\psi$ descends to a map on $G_2$ is the most complicated one. In preparation, we start with the following calculations in $G_1$.

\begin{lem}
\label{lem:qchange}
For each $i$, let $p_i:=t_i t_{i-1}\ldots t_1\in G_1$ (so $q_i = p_1p_2\ldots p_i$). Then
$$q_i = q_{i+1}p_{i+1}^{-1} = p_{i+1}q_{i+1},$$
and
$$q_i = q_{i-1}p_{i} = p_{i}^{-1}q_{i-1}.$$
\end{lem}
\begin{proof}
The fact that $q_i = q_{i+1}p_{i+1}^{-1}$ is trivial. The fact that $q_i = p_{i+1}q_{i+1}$ follows by an easy induction:
$$p_{i+1}q_{i+1} = (t_{i+1} p_i) (q_i p_{i+1}) = t_{i+1} q_{i-1}p_{i+1} = q_{i-1} t_{i+1} p_{i+1} = q_{i-1} p_i = q_i.$$
The remaining identities follow from that ones just shown.
\end{proof}

Now we state two technical lemmas for the proof.

\begin{lem}
\label{lem:tech1}
Suppose $0 < i \leqslant k < l \leqslant j$. The following identity holds in $G_1$:
\begin{multline*}
q_{j-1} q_{j-i} q_{j-1} q_{l-1} q_{l-k} = q_{j - (k-i) - 1}\cdot\\
\Bigl((t_{j-(k-i)} t_{j-(k-i)-1}\ldots t_{l-k+2}) (t_{j - (k-i) + 1} t_{j - (k-i)}\ldots t_{l-k+3})\ldots (t_{j-1} t_{j-2}\ldots t_{l-i+1})\Bigr)\cdot\\
\Bigl((t_{j-i+1} t_{j-i}\ldots t_{l-i+2}) (t_{j-i+2} t_{j-i+1}\ldots t_{l-i+3})\ldots (t_{j-1} t_{j-2}\ldots t_{l})\Bigr).
\end{multline*}
\end{lem}

\begin{lem}
\label{lem:tech2}
Suppose $0 < i \leqslant k < l \leqslant j$. The following identity holds in $G_1$:
\begin{multline*}
q_{l-1} q_{j-1} q_{j-i}q_{j-1} =
(p_{l}p_{l+1}\ldots p_{j-1})(p_{j-i+1}p_{j-i+2}\ldots p_{j-1})\cdot\\
(p_{j - (k-i) -1}^{-1}p_{j - (k-i) -2}^{-1}\ldots p_{l-k+1}^{-1})q_{l-k}q_{j-(k-i) -1}
\end{multline*}
\end{lem}

\begin{proof} [Proof that $\psi(q_{[i,j]} q_{[k,l]} q_{[i,j]}) = \psi(q_{[j-k+i,j-l+i]})$]
We need to show that 
\[q_{j-1} q_{j-i} q_{j-1} q_{l-1} q_{l-k} q_{l-1} q_{j-1} q_{j-i}q_{j-1} = q_{j-k+i-1} q_{l-k} q_{j-k+i-1}\]
The left hand side is just the product of the left hand sides in the two preceding lemmas.

We will now look in more detail at what happens when we multiply the right hand sides. Note that 
\[t_{j-1}t_{j-2}\ldots t_{l} p_l p_{l+1}\ldots p_{j-1} = p_{l-1}p_{l}\dots p_{j-2}.\]
The right hand side of Lemma \ref{lem:tech1} consists of three lines; by repeating the above calculation we conclude that 
the third line multiplied by $p_{l}p_{l+1}\dots p_{j-1}$ yields 
\[p_{l-i+1}p_{l-i+2}\dots p_{j-i}.\]
By same logic, multiplying the second line by $p_{l-i+1}p_{l-i+2}\dots p_{j-1}$ yields 
\[p_{l-k+1}p_{l-k+2}\dots p_{j-(k-i)-1}.\]
This neatly cancels the sequence of $p^{-1}$'s, finishing the proof.
\end{proof}

\begin{proof}[Proof of Lemma \ref{lem:tech1}]
Let $(*)$ stand for the left hand side of the lemma, namely for $q_{j-1} q_{j-i} q_{j-1} q_{l-1} q_{l-k}$. Repeatedly using Lemma \ref{lem:qchange} (and the fact that $q$'s are involutions) yields
\begin{multline*}
(*) = (q_{j-(k-i)-1}p_{j-(k-i)}p_{j-(k-i)+1}\ldots p_{j-1})\xout{q_{j-i}}(\xout{q_{j-i}}p_{j-i+1}p_{j-i+2}\ldots p_{j-1})\\(p_{l-1}^{-1}p_{l-2}^{-1}\ldots p_{l-k+1}^{-1} \xout{q_{l-k}})\xout{q_{l-k}}
\end{multline*}
Now notice that $p_{j-1}p_{l-1}^{-1} = t_{j-1}t_{j-2}\ldots t_l$ commutes with $p_{l-2}$, $p_{l-3}$, $p_{l-k+1}$. Repeating this argument gives
\begin{multline*}
(*) = q_{j-(k-i)-1}(p_{j-(k-i)}p_{j-(k-i)+1}\ldots p_{j-1})(p_{l-i}^{-1}p_{l-i-1}^{-1}\ldots p_{l-k+1}^{-1})\\
\Bigl((t_{j-i+1}t_{j-i}\ldots t_{l-i+2})(t_{j-i+2}t_{j-i+1}\ldots t_{l-i+3}) \ldots(t_{j-1}t_{j-2}\ldots t_l)\Bigr)
\end{multline*}
Doing the same thing for the first and second parentheses group gives
\begin{multline*}
(*) = q_{j-(k-i)-1}\\
\Bigl((t_{j-(k-i)}t_{j-(k-i)-1}\ldots t_{l-k+2})(t_{j-(k-i)+1}t_{j-(k-i)}\ldots t_{l-k+3}) \ldots(t_{j-1}t_{j-2}\ldots t_{l-i+1})\Bigr)\\
\Bigl((t_{j-i+1}t_{j-i}\ldots t_{l-i+2})(t_{j-i+2}t_{j-i+1}\ldots t_{l-i+3}) \ldots(t_{j-1}t_{j-2}\ldots t_l)\Bigr)
\end{multline*}
as desired.
\end{proof}

\begin{proof}[Proof of Lemma \ref{lem:tech2}]
Each line in the calculation below consists of repeated use of Lemma \ref{lem:qchange} to increase/decrease indices of $q$'s by pulling out appropriate $p$'s.
$$\begin{array}{rcl}\medskip
& & q_{l-1} q_{j-1} q_{j-i} q_{j-1} \\\medskip
   & = & q_{l-1} q_{j-1} q_{j-i} (p_{j-1}^{-1}p_{j-2}^{-1}\ldots p_{j-(k-i)}^{-1}) q_{j- (k-i) -1}\\\medskip
   & = & q_{l-1} (p_{j-1}^{-1}p_{j-2}^{-1}\ldots p_{j-i+1}^{-1}) \xout{q_{j-i}q_{j-i}} (p_{j-1}^{-1}p_{j-2}^{-1}\ldots p_{j-(k-i)}^{-1}) q_{j- (k-i) -1}\\\medskip
   & = & (p_lp_{l+1}\ldots p_{j-1})q_{j-1} (p_{j-1}^{-1}p_{j-2}^{-1}\ldots p_{j-i+1}^{-1})(p_{j-1}^{-1}p_{j-2}^{-1}\ldots p_{j-(k-i)}^{-1}) q_{j- (k-i) -1}\\\medskip
   & = & (p_lp_{l+1}\ldots p_{j-1})q_{j-i} (p_{j-1}^{-1}p_{j-2}^{-1}\ldots p_{j-(k-i)}^{-1}) q_{j- (k-i) -1}\\\medskip
   & = & (p_lp_{l+1}\ldots p_{j-1})(p_{j-i+1}p_{j-i+2}\ldots p_{j-1})q_{j-1} (p_{j-1}^{-1}p_{j-2}^{-1}\ldots p_{j-(k-i)}^{-1}) q_{j- (k-i) -1}\\\medskip
   & = & (p_lp_{l+1}\ldots p_{j-1})(p_{j-i+1}p_{j-i+2}\ldots p_{j-1})q_{j-(k-i)-1}  q_{j- (k-i) -1}\\\medskip
   & = & (p_lp_{l+1}\ldots p_{j-1})(p_{j-i+1}p_{j-i+2}\ldots p_{j-1})\\\medskip
	 &   & \qquad(p_{j - (k-i) -1}^{-1}p_{j - (k-i) -2}^{-1}\ldots p_{l-k+1}^{-1})q_{l-k}  q_{j- (k-i) -1}
\end{array}$$
\end{proof}

Now we show that $\phi$ descends to $G_1$

\begin{proof} [Proof that $\phi(t_i^2)=I$.]
 The proofs of $\phi(t_1)^2=I$ and of $\phi(t_2)^2=I$ are easy. For $i>2$,
\begin{multline*}
\phi(t_i^2) = q_{[1,i]}q_{[1,i+1]}q_{[1,i]}q_{[1,i-1]}q_{[1,i]}q_{[1,i+1]}q_{[1,i]}q_{[1,i-1]}\\
 = q_{[1,i]}q_{[1,i+1]}q_{[2,i]} q_{[1,i+1]}q_{[1,i]}q_{[1,i-1]}
 = q_{[1,i]} q_{[2,i]} q_{[1,i]} q_{[1,i-1]} = q_{[1,i-1]} q_{[1,i-1]} = I.
\end{multline*}
\end{proof}

Before the next proof, we need the following lemma.

\begin{lem} \label{lem:comm}
 The following identities are satisfied in $G_2$.
 For $a < b \leq c$ we have $$q_{[a,b]} q_{[1,c]} q_{[1,c+1]} = q_{[1,c]} q_{[1,c+1]} q_{[a+1,b+1]}.$$
 For $1 < a < b \leq c$ we have $$q_{[a,b]} q_{[1,c]} q_{[1,c-1]} = q_{[1,c]} q_{[1,c-1]} q_{[a-1,b-1]}.$$
\end{lem}

\begin{proof}
 We have 
 $$q_{[a,b]} q_{[1,c]} q_{[1,c+1]} =  q_{[1,c]} q_{[c-b+1, c-a+1]} q_{[1,c+1]} = q_{[1,c]} q_{[1,c+1]} q_{[a+1,b+1]}.$$
We also have 
 $$q_{[a,b]} q_{[1,c]} q_{[1,c-1]} = q_{[1,c]} q_{[c-b+1, c-a+1]} q_{[1,c-1]} = q_{[1,c]} q_{[1,c-1]} q_{[a-1,b-1]}.$$
\end{proof}

\begin{proof} [Proof that $\phi(t_i t_j) = \phi(t_j t_i)$.]
Assume $i+2 \leq j$. For $i \not = 1,2$ we have 
$$\phi(t_i t_j) = q_{[1,i]} q_{[1,i+1]} q_{[1,i]} q_{[1,i-1]} (q_{[1,j]} q_{[1,j+1]}) (q_{[1,j]} q_{[1,j-1]}).$$
Applying Lemma \ref{lem:comm} for $c=j$, $a=1,2$ and $b = i-1, i, i+1$ we have 
\[\begin{array}{rcl}\medskip
\phi(t_i t_j)
 &=& q_{[1,i]} q_{[1,i+1]} q_{[1,i]} (q_{[1,j]} q_{[1,j+1]}) q_{[2,i]} (q_{[1,j]} q_{[1,j-1]})\\\medskip
 &=& q_{[1,i]} q_{[1,i+1]} q_{[1,i]}  (q_{[1,j]} q_{[1,j+1]}) (q_{[1,j]} q_{[1,j-1]}) q_{[1,i-1]}\\\medskip 
 &=& \ldots\\\medskip  
 &=&(q_{[1,j]} q_{[1,j+1]}) (q_{[1,j]} q_{[1,j-1]}) q_{[1,i]} q_{[1,i+1]} q_{[1,i]} q_{[1,i-1]},
\end{array}\]
as desired. 

Essentially the same proof works for $i=1,2$.
\end{proof}

\begin{proof} [Proof that $\phi$ and $\psi$ are inverses.]
We have 
\[\psi(\phi(t_1)) = \psi(q_{[1,2]}) = q_1^3 =q_1 = t_1\]
\[\psi(\phi(t_2)) = \psi(q_{[1,2]}q_{[1,3]}q_{[1,2]}) = q_1^3q_2^3q_1^3 =q_1q_2q_1 = t_2\]
\[\begin{array}{rcl}\medskip
\psi(\phi(t_i)) & = & \psi(q_{[1,i]}q_{[1,i+1]}q_{[1,i]}q_{[1,i-1]})\\ \medskip 
 & = & q_{i-1}q_iq_{i-1}q_{i-2} \\ \medskip
 & = & p_i q_{i-1}q_{i-2}\\ \medskip
 & = & t_i p_{i-1} q_{i-1}q_{i-2}\\ \medskip
 & = & t_i q_{i-2}q_{i-2}\\ \medskip
 & = & t_i
\end{array}\]
Thus $\phi$ is injective and is the right inverse of $\psi$. But it is easy to see that every $q_{[i,j]}$ is in the image of $\phi$, so it 
is a bijection. Thus $\phi$ and $\psi$ are indeed inverses.
\end{proof}

In the introduction we stated the following Theorem \ref{thm:rel2}.
\begin{theorem*} 
 The relations $$t_i^2=I, \;\; t_i t_j = t_j t_i \text{ if } |i-j|>1 \text{  and  } (t_i q_{k-1}q_{k-j}q_{k-1})^2 = I \text{ for } i+1 < j$$ are equivalent to the relations 
 $$q_{[i,j]}^2 = I, \;\; q_{[i,j]} q_{[k,l]} q_{[i,j]} = q_{[i+j-l,i+j-k]} \text{ if } i \leq k < l \leq j,$$ $$\text{  and  } q_{[i,j]} q_{[k,l]} = q_{[k,l]} q_{[i,j]} \text{ if } j<k.$$
\end{theorem*}

Part of the needed implications follows from Theorem \ref{thm:rel1}. We prove the remaining ones. 

\begin{proof} [Proof that $\psi(q_{[i,j]} q_{[k,l]}) = \psi(q_{[k,l]} q_{[i,j]})$.]
By definition, $\psi(q_{[i,j]}) = q_{j-1} q_{j-i} q_{j-1}$ is an expression of $t_1$ through $t_{j-1}$. Each such $t_i$ commutes with $\psi(q_{[k,l]})$ because $i+1 \leq j < k$. Then so does $\psi(q_{[i,j]})$.
\end{proof}

\begin{proof} [Proof that $\phi((t_i q_{k-1}q_{k-j}q_{k-1})^2) = I$.]
By definition, $\phi(t_i) = q_{[1,i]} q_{[1,i+1]} q_{[1,i]} q_{[1,i-1]} \text{ for } i>2$. Each of $q_{[1,i-1]}$, $q_{[1,i]}$, and $q_{[1,i+1]}$ commutes with  $q_{[j,k]} = \phi(q_{k-1}q_{k-j}q_{k-1})$ because $i+1 < j$. Then so does $\phi(t_i)$. The cases $i=1$ and $i=2$ are similar.
\end{proof}

\bibliographystyle{plain}
\bibliography{BK}

\end{document}